\documentclass{amsart}
\usepackage{amscd}
\usepackage{amssymb}

\def\can{\mathrm{can}}
\def\cd{\operatorname{cd}}

\def\pt{\mathrm{pt}}

\def\BvS{Buchweitz and van Straten }
\def\BvSns{Buchweitz and van Straten}
\def\and{ \, \, \mathrm{and} \, \,   }
\def\alg{{\mathrm{alg}}}

\def\link{\lk}
\def\lk{\operatorname{link}}
\def\o{\overline}
\def\id{{\text{id}}}
\def\hen{{\mathrm{hen}}}
\def\sh{{\mathrm{sh}}}

\def\ov{\overline}

\def\a{\alpha}
\def\b{\beta}
\def\g{\gamma}

\newcommand{\pd}[2]{{\scriptstyle{\frac{\partial #1}{\partial #2}}}}

\newcommand{\adi}[1]{[{\scriptstyle{\frac{1}{#1}}}]}

\def\len{\length}
\newcommand{\length}{\operatorname{length}}

\newcommand{\xra}[1]{\xrightarrow{#1}}
\newcommand{\xla}[1]{\xleftarrow{#1}}

\input xy
\xyoption{all}

\newcommand{\ai}[1]{{\left[ {\scriptstyle \frac{1}{#1}} \right]}}

\def\ob{\operatorname{ob}}

\def\cube#1#2#3#4#5#6#7#8{
& #5 \ar[rr] \ar[dl] \ar@{-}[d] && #6 \ar[dd] \ar[dl] \\
#1 \ar[rr] \ar[dd]  & \ar[d] & #2 \ar[dd] \\
& #7 \ar@{-}[r] \ar[dl] & \ar[r] & #8 \ar[dl] \\
#3 \ar[rr] && #4 \\
}

\def\ip#1{\left< #1 \right>}
\def\ad#1{{\scriptstyle \left[ \frac{1}{#1} \right]}}

\def\chr{\operatorname{char}}

\def\sm{\setminus}

\def\etale{\'etale~}

\def\image{\operatorname{im}}
\def\im{\image}
\def\ker{\operatorname{ker}}

\def\cP{\mathcal P}
\def\cE{\mathcal E}

\def\cA{\mathcal A}

\def\cF{\mathcal F}
\def\cG{\mathcal G}
\def\cO{\mathcal O}

\def\cS{\mathcal S}

\def\cM{\mathcal M}

\def\top{\mathrm{top}}
\def\cK{\mathcal K}
\def\cKU{{\mathcal{KU}}}

\def\Ktop{K^{\mathrm{top}}}
\def\Gtop{G^{\mathrm{top}}}

\def\coker{\operatorname{coker}}

\def\dm{\operatorname{dim}}

\def\Tor{\operatorname{Tor}}

\def\Mat{\operatorname{Mat}}

\def\Spec{\operatorname{Spec}}
\def\mSpec{\operatorname{mSpec}}

\def\map#1{{\buildrel #1 \over \lra}} 
 
\def\lra{\longrightarrow}
\def\into{\rightarrowtail}
\def\onto{\twoheadrightarrow}

\def\o#1{{\overline{#1}}}
\def\t#1{{\widetilde{#1}}}

\DeclareMathOperator*{\colim}{colim}

\newcommand{\A}{\mathbb{A}}

\newcommand{\C}{\mathbb{C}}
\newcommand{\Z}{\mathbb{Z}}
\newcommand{\N}{\mathbb{N}}

\newcommand{\fp}{{\mathfrak p}}
\newcommand{\fm}{{\mathfrak m}}
\newcommand{\fn}{{\mathfrak n}}
\newcommand{\fq}{{\mathfrak q}}

\numberwithin{equation}{section}

\theoremstyle{plain} 
\newtheorem{thm}[equation]{Theorem}
\newtheorem{cor}[equation]{Corollary}
\newtheorem{lem}[equation]{Lemma}
\newtheorem{prop}[equation]{Proposition}

\newtheorem{assumptions}[equation]{Assumptions}

\newtheorem{ex}[equation]{Example}

\newtheorem{defn}[equation]{Definition}

\theoremstyle{remark}
\newtheorem{rem}[equation]{Remark}

\newtheorem*{ack}{Acknowledgements}

\def\smm{\setminus \fm}

\def\del{\partial}

\def\sm{\setminus}
\def\et{\text{\'et}}
\def\ttheta{\widetilde{\theta}}

\def\sstar{\star_{\text{Sh}}}
\def\gstar{\star_{\text{Gr}}}
\def\mstar{\star_{\text{Mi}}}

\begin{document}

\title{On the vanishing of Hochster's $\theta$ invariant}

\date{\today}

\author{Mark E. Walker} 
\address{Department of Mathematics, University
  of Nebraska, Lincoln, NE 68588} 
\email{mwalker5@math.unl.edu}
\thanks{The author was supported in part by National Science Foundation Award DMS-0966600}

\begin{abstract} 
Hochster's theta invariant is defined for a pair of finitely generated modules on a hypersurface ring having only an isolated singularity. Up to a sign, it
agrees with the Euler invariant of a pair of matrix factorizations.

Working over the complex numbers, 
Buchweitz and van Straten have established an interesting connection between Hochster's $\theta$-invariant  and the classical linking form on the link of the singularity.
In particular, they establish the vanishing of the $\theta$ invariant if the hypersurface is even-dimensional by exploiting the fact that the (reduced) cohomology of the
Milnor fiber is concentrated in  odd degrees in this situation.

In this paper, we give purely algebraic versions of some of these results. In particular, we establish the vanishing of the $\theta$ invariant for isolated
hypersurface singularities of odd dimension in characteristic $p > 0$, under some mild extra assumptions. This confirms, in a large number of cases,  a
conjecture of Hailong Dao.
\end{abstract}


\maketitle

\tableofcontents

\section{Introduction}
In this paper, a {\em hypersurface} will refer to a ring $R$ that can be
expressed as a quotient of a regular (Noetherian) ring $Q$ by a
non-zero-divisor $f$; i.e., $R = Q/f$. (We do not require $Q$ to be local.)
If $M,N$ are finitely generated
$R$-modules, then using the standard long exact sequence
$$
\cdots \to \Tor_Q^{i-1}(M,N) \to \Tor_R^{i}(M,N)  \to
\Tor_R^{i-2}(M,N) \to \Tor_Q^{i-2}(M,N) \to \cdots
$$
and the fact that $\Tor_Q^j(M,N) = 0$ for $j \gg 0$ since $Q$ is regular, we conclude that $\Tor_R^*(M,N)$ is 
eventually two-periodic: there is an isomorphism
$$
\Tor_R^i(M,N) \xra{\cong} \Tor_R^{i+2}(M,N), \, \text{ for $i \gg 0$.}
$$

If we assume, in addition, that $\Tor_R^i(M,N)$ has finite length for all $i \gg
0$, then Hochster's {\em theta invariant} \cite{Hochster}
of the pair $(M,N)$ is defined to be the integer
$$
\theta^R(M,N) = \length \Tor_R^{2i}(M,N) - \length \Tor_R^{2i+1}(M,N),
\, \text{ for $i \gg 0$.}
$$
The modules $\Tor^R_i(M,N)$ will have finite length for $i \gg 0$ if, for example, there exists a
maximal ideal $\fm$ of $R$ such that one of the modules, say $N$, is locally of finite 
projective dimension on the
punctured spectrum 
$\Spec(R) \setminus \fm$. In this situation, $N$
determines a coherent sheaf on the quasi-affine scheme
$\Spec(R) \smm$ that admits a finite resolution by locally free coherent
sheaves, and hence $N$ determines a class in $[N] \in K_0(\Spec(R) \smm)$. The module
$M$, of course, determines a class $[M] \in G_0(R)$. 
Hochster proves \cite[1.2]{Hochster} that $\theta$ is bi-additive for such
pairs of modules and hence determines a pairing
$$
G_0(R) \times K_0(\Spec R \setminus \fm) \to \Z.
$$
Moreover, $\theta(R/\fm, -)$ is identically zero
and, since $G_0(R) \to G_0(\Spec(R) \smm)$ is surjective with kernel
generated by $[R/\fm]$, we obtain an induced pairing
\begin{equation} \label{deftheta}
\theta = \theta_{(Q,\fm, f)}: G_0(\Spec(R) \smm) \times K_0(\Spec(R) \setminus \fm) \to \Z.
\end{equation}
We will refer to this as {\em Hochster's $\theta$ pairing associated to the data $(Q,f,\fm)$}.

In the paper \cite{BvS}, Buchweitz and van Straten relate Hochster's $\theta$
pairing for isolated hypersurface singularities of the form $R =
\C\{x_0, \dots, x_n\}/(f)$, where $\C\{x_0, \dots, x_n\}$ denotes the
ring of convergent power series,
to the linking form on the link of the singularity. Using this
relationship, 
they also prove the $\theta$ pairing vanishes when $\dm(R)$ is even.

The goal of this paper is to give a purely algebraic interpretation of some of the
results of \BvSns, ones which are also valid in characteristic $p>0$. In particular,
we obtain the vanishing of $\theta$ for isolated hypersurface singularities of even dimension in
all characteristics for a large number of rings, confirming in many
cases
a conjecture of H.~Dao \cite[3.15]{DaoDecent}. We refer the reader to Corollaries \ref{CorMain} and \ref{Cor114} for the most general statements, but an
important special case of our results is given by the following Theorem: 

\begin{thm} \label{introthm}
Let $k$ be a perfect field, $Q$ a finitely generated and regular $k$-algebra, and $f \in Q$ a non-zero-divisor.
Assume the associated morphism of affine varieties 
$$
f: \Spec(Q) \to \A^1_k
$$
has only isolated singularities.
If $n = \dm(Q/f)$ is even, then  $\theta^R(M,N) = 0$ for all finitely generated $R$-modules $M$ and $N$.
\end{thm}

The Theorem may be extended easily to allow for localizations of smooth algebras, and thus justifies examples such as the following:

\begin{ex} \label{IntroEx} Let $k$ be a perfect field, let $f \in k[x_0, \dots, x_n]$ be a polynomial in $n+1$ variables contained in 
$\fm = \langle x_0, \dots, x_n \rangle$, and 
assume
$\ip{\pd{f}{x_0}, \dots, \pd{f}{x_n}}_\fm$
is an $\fm$-primary ideal of $k[x_0, \dots, x_n]_\fm$.
Then $\theta^{R}(M,N) = 0$, where  $R := k[x_0, \dots, x_n]_\fm/f$, 
for all pairs of finitely generated $R$-modules.
\end{ex}

One technical aspect of this paper may be of independent interest: in
Section \ref{sec:Sherman}, we prove a slightly weakened form of a
conjecture of C. Sherman \cite[\S2]{ShermanConnII} concerning his ``star pairing'' in algebraic $K$-theory.

\begin{ack} I am grateful to Michael Brown, Ragnar Buchweitz, Jesse Burke, Olgur Celikbas, Hailong Dao, Dan Grayson, Paul Arne {\O}stv{\ae}r, Andreas
  Rosenschon,  Clayton Sherman, and Duco van Straten for conversations about the topics of this paper.
\end{ack}

\section{Overview}
We give an overview of this paper by comparing the details we use  with those occurring in the work of 
 Buchweitz and van Straten.

\subsection{The Milnor fibration}
Given a power
series $f \in \C\{x_0, \dots, x_n\}$ with positive radius of convergence, we interpret $f$ as defining a holomorphic function defined on 
an open  neighborhood of $0$ in $\C^{n+1}$.
Let $\o{B}$, the {\em Milnor ball}, be a closed ball of radius $\epsilon$ centered at the 
origin of $\C^{n+1}$
and let $D$ be an open disk of radius $\delta$ centered at the origin in the complex plane 
$\C$. Here, $\epsilon$ is chosen first and to be sufficiently small, and $\delta$ is chosen to be sufficiently smaller than $\epsilon$: $0 < \delta \ll \epsilon
\ll 1$.
Set $X = f^{-1}(D) \cap \o{B}$ and write also $f$ for the restricted function $f: X \to D$.

Let us assume that the fiber $X_0$ of $f$ over $0 \in D$ is an
isolated singularity at the origin. It follows that $X_0$ is homeomorphic to the
cone over the {\em link} $L$ of the singularity, defined as 
$L := X_0
\cap S$, where $S := \partial \o{B} \cong S^{2n+1}$ is the 
{\em Milnor sphere}.
The link $L$ is a smooth orientable  manifold of real dimension
$2n-1$.

The map away from the singular locus,
$$
X^* := (X \setminus X_0) \to 
(D \setminus 0) =: D^*
$$
induced by $f$ 
is a fibration, called the {\em Milnor
  fibration}. For any $t \ne 0$ in $D$, the fiber $X_t$ of
$f$ over $t$ is a smooth manifold with boundary of (real) dimension $2n$. Up to
diffeomorphism, $X_t$ is 
 independent of $t$, and it is called
the {\em Milnor fiber} of the singularity. 
A key result
of Milnor \cite{MilnorFiber} gives that the Milnor fiber has the homotopy type of a bouquet of
$n$-dimensional spheres; the number of spheres is the {\em Milnor number}, often written as $\mu$. 
In particular, the reduced singular cohomology of the Milnor fiber, 
$\tilde{H}^*(X_t)$, is a free abelian group of rank $\mu$ concentrated in degree
$n$.

The restriction of $f$ to $S \cap X$ is a fibration and thus, up to diffeomorphism, we may identify the boundary of the Milnor fiber,
$\partial X_t = \left(f|_{X \cap S}\right)^{-1}(t)$, with 
the link $L = \left(f|_{X \cap S}\right)^{-1}(0)$. In particular, for each $t \in D^*$ there is a continuous map
$$
\rho_t: L \to X^*
$$
given by composing the diffeomorphism  $L \cong \partial X_t$  with the inclusion $\partial 
X_t \subseteq X^*$.

\subsection{The vanishing of the pairing when $n$ is even}
With the above notation set, we sketch the proof of \BvS for the vanishing of $\theta$ when 
$n$ is even.

Associated to $MCM$ $\C\{x_0, \dots, x_n\}/f$-modules $M$ and $N$, \BvS  associate 
classes in topological $K$-theory:
$$
\alpha(M) \in K^1_\top(X^*)
\and [N]_\top \in K^0_\top(L).
$$
The class $[N]_\top$ is given by the evident restriction of the coherent sheaf associated to $N$ to $L$ and the class 
$\alpha(M)$
is built from a {\em matrix factorization} representation of $M$.
They prove \cite[4.2]{BvS} that
$$
\theta(M,N) = \chi_\top(\rho_t^*(\alpha(M))  \cup [N]_\top).
$$
Here, $\rho_t^*: K^1_\top(X^*) \to K^1_\top(L)$ is the map induced by the map 
$\rho_t: L \into X^*$ defined above, $\cup$ is the product operation for the ring $K_\top^*$, and 
$\chi_\top: K^1_\top(L) \to K^0_\top(\pt) \cong \Z$ is induced by
push-forward. (Recall $L$ is odd dimensional and so $\chi_\top$ switches the parity of 
degrees.)

Notice that $\rho_t^*$ factors through $K^1_\top(X_t)$, since $\rho_t$ factors through $X_t$ 
by its very construction.
When $n$ is even, we have $K^1_\top(X_t) = 0$, since $X_t$ is a bouquet of $n$-dimensional 
spheres by Milnor's Theorem \cite{MilnorFiber}. 
It follows that $\rho_t^*(\alpha(M)) = 0$ and hence
$$
\theta(M,N) = 
\chi_\top(\rho_t^*(\alpha(M))  \cup [N]_\top) = 0.
$$

\subsection{The algebraic analogue of the Milnor fibration}
The algebraic analogue of the Milnor fibration is constructed as follows:
Assume $V$ is a Henselian dvr with algebraically closed residue field
$k$ and field of fractions $F$. For example, $V$ could be $k[[t]]$ with $k = \overline{k}$.
The affine scheme $\Spec(V)$ is the analogue of $D$, a small open disk in the complex plane. 
Let $Q$ be a flat $V$-algebra of finite type and
assume the associated 
morphism of affine schemes
$$
\Spec(Q) \to \Spec(V)
$$
is smooth except at a single point $\fm \in \Spec(Q)$ necessarily belonging to the closed fiber. Let 
$Q^\hen_\fm$ denote the Henselization of $Q$ at $\fm$. The affine scheme $\Spec(Q^\hen_\fm)$ is the analogue of $X = f^{-1}(D) \cap \overline{B}$ in the notation above.
Thus the morphism 
$$
X^\alg := \Spec(Q^\hen_\fm) \to \Spec(V) =: D^\alg
$$
is the algebraic analogue of the analytic map $f: X \to D$ considered above. 
The generic fiber of $X^\alg \to D^\alg$, 
$$
X^*_\alg := \Spec(Q^\hen_\fm \otimes_V F) \to \Spec(F) =: D^*_\alg,
$$
is the analogue of the Milnor fibration, and the geometric generic fiber,
$$
X_t^\alg := \Spec(Q^\hen_\fm \otimes_V \o{F}),
$$
where $\o{F}$ is the algebraic closure of $F$, is the analogue of the Milnor fiber.

\begin{rem} It is important to be aware that 
$Q^\hen_\fm \otimes_V \o{F}$ need not be a Noetherian ring.
\end{rem}

The singularity is the closed fiber of $X^\alg \to D^\alg$,
$$
X_0^\alg :=  \Spec(R^\hen_\fm) = \Spec(Q^\hen_\fm/f),
$$
where $f$ denotes the image in $Q$ of a specified uniformizing local parameter $t \in V$ and we set $R = Q/f$.
The algebraic analogue of the link is the punctured spectrum of the singularity:
$$
L^\alg := \Spec(R^\hen_\fm) \sm \fm.
$$

We summarize these constructions and introduce a few more analogies in the following table.

\begin{center}
{\small
\begin{tabular}{| p{2.5in} | p{2.5in}|}
\hline
Geometric Notion & Algebraic Analogue \\
\hline \hline 
$D$, a small open disc in complex plane & $D^\alg := \Spec(V)$, where $V$ is a Henselian 
dvr with uniformizing parameter $t$, algebraically closed residue
field $k$, and field of fractions $F$ \\
\hline
convergent power series $f \in \C\{x_0, \dots, x_n\}$ representing an isolated singularity &
a flat, finite type ring map $V \to Q$ sending $t \in V$ to $f \in Q$ that is smooth away from 
some point $\fm$ in the closed fiber \\
\hline
$X$, the intersection of a small closed ball in $\C^{n+1}$ with $f^{-1}(D)$ & $X^\alg = 
\Spec(Q^\hen_\fm$) \\
\hline
Milnor fibration $X^* \to D^*$ & the generic fiber 
of $X^\alg \to D^\alg$; namely, $\Spec(Q^\hen \otimes_V F) \to \Spec(F)$\\
\hline
the Milnor fiber $X_t := f^{-1}(t)$, $t \ne 0$ & 
the geometric generic fiber of $X^\alg \to D^\alg$; namely, 
$X_t^\alg := \Spec(Q^\hen_\fm \otimes_V \o{F})$ \\
\hline
the singularity $X_0 = f^{-1}(0)$ & 
$X_0^\alg := \Spec(R^\hen_\fm)$, where $R = Q/f$ (i.e., the 
closed fiber of $X^\alg \to D^\alg$) \\
\hline
the link $L$ & the punctured spectrum \phantom{XXXXXXXXX} \hfill \linebreak $L^\alg := \Spec(R^{\hen}_\fm) \setminus \fm$ \\
\hline
Milnor ball $\o{B}$ & $\o{B}^\alg = \Spec(Q^\hen_\fm)$ \\
\hline
Milnor sphere $S=\partial \o{B}$  & $S^\alg = \Spec(Q^\hen_\fm) \setminus \fm$ \\
\hline
$\rho_t^*: K^1_\top(X^*) \to K^1_\top(L)$ & the ``specialization map'' in $K$-theory with 
finite coefficients. \\
\hline
\end{tabular}
}
\end{center}

\begin{rem} Note that the algebraic analogue of the Milnor ball  and 
the algebraic analogue of $X$ coincide, in
  contrast with what occurs in the analytic setting. The analogy could be slightly improved if 
one does not assume from the start that $V$ is Henselian.
Then the algebraic analogue the Milnor ball remains $\Spec(Q^\hen_\fm)$ but the algebraic 
analogue of $X \to D$ would become 
$\Spec(Q^\hen_\fm \otimes_V V^\hen) \to \Spec(V^\hen)$, where $V^\hen$ denotes the 
Henselization of $V$ at
its unique maximal ideal. For this paper,  however, this more general construction has no 
advantage.
\end{rem}

With this notation fixed,  let us sketch our proof of the vanishing of the $\theta$ pairing for 
$R = Q/f$. 

Given finitely generated MCM $R$-modules $M$ and $N$,  we define classes 
$$
\alpha^{alg}(M) \in K_1(X^*_\alg) = K_1(Q\adi{f})
$$
and
$$
[N] \in K_0(L^\alg) = K_0(\Spec(R_\fm) \sm \fm),
$$
where $[N]$ is defined as the image of the class of $N$ in $G_0(R)$ under the restriction map 
$G_0(R) \to G_0(\Spec(R) \sm \fm) \cong
K_0(\Spec(R) \sm \fm)$ and
$\alpha^\alg(M) = [A] \in K_1(Q \adi{f})$ for any matrix factorization $(A,B)$ representation 
of $M$. 
Our first key result (see Corollary \ref{Cor3}) is the equation
\begin{equation} \label{E819}
\theta^R(M,N) = \chi\left(\partial\left(\alpha(M) \cup [f]\right) \cup [N]\right).
\end{equation}

Let us explain the components of this formula. Since $f$ is a unit of $Q\adi{f}$, it determines 
a class $[f] \in K_1(Q\adi{f})$. The symbol 
$\cup$ denotes the product operation in $K$-theory.  The map
$\partial: K_2(Q\adi{f}) \to K_1(\Spec(R) \sm \fm)$ is the boundary map in the long exact 
localization sequence associated to the closed
subscheme $L^\alg = \Spec(R) \sm \fm$ of $S^\alg = \Spec(Q) \sm \fm$.
(Since $Q\adi{f}$ and $\Spec(R)
\sm \fm$ are regular, we use, as we may,  $K$-theory in place of $G$-theory.) 
Finally,
$$
\chi: K_1(\Spec(R) \sm \fm) \to \Z.
$$
is the composition of the boundary map 
$$
K_1(\Spec(R) \sm \fm) \to
K_0(\Spec(R/\fm))
$$
with the canonical isomorphism $K_0(\Spec(R/\fm)) \cong \Z$, sending $[R/\fm]$ to $1$.

All of these facts remain true if we use $K$-theory with $\Z/\ell$ coefficients, where
$\ell$ is a prime distinct from the characteristic of $Q/\fm$ (and, for technical reasons, 
$\ell \geq 5$). Moreover, we may replace algebraic $K$-theory with (a version
of) \etale $K$-theory with $\Z/\ell$ coefficients, which we write as $\Ktop_*(-, \Z/\ell)$. 

A portion of the formula \eqref{E819}, using $\Ktop_*(-, \Z/\ell)$ instead of $K_*(-)$, is given by the 
mapping 
$$
\Ktop_1(X^*_\top, \Z/\ell) = \Ktop_1(Q\adi{f}, \Z/\ell) \to  \Ktop_1(\Spec(R) \sm \fm,  
\Z/\ell) = \Ktop_1(L^\alg, \Z/\ell)
$$
that sends $\a$ to $\partial(\a \cup [f])$. We interpret this map, under some additional 
hypotheses, as a ``specialization map'' in $K$-theory, and  it is the analogue of the
map $\rho_t^*$ occurring in the work of \BvSns.
The next key result is that this specialization map 
factors through 
$$
\Ktop_1(X_t^\alg, \Z/\ell) = \Ktop_1(Q^\hen_\fm \otimes_V \o{F}, \Z/\ell),
$$ 
This factorization is the analogue of the factorization of $\rho_t^*$ through the $K$-theory 
of the Milnor fiber used in the work of \BvSns.

Finally, we combine a Theorem of  Rosenschon-{\O}stv{\ae}r 
\cite[4.3]{OR} (which is a generalization of a celebrated theorem of Thomason \cite[4.1]
{Thomason}) and a Theorem of  
Illusie \cite[2.10]{Illusie}  to
prove 
$$
\Ktop_1(X_t^\alg, \Z/\ell) = 0
$$
if $n$ is even (and some additional mild hypotheses hold). Illusie's Theorem is in fact a good 
analogue of Milnor's Theorem, that the Milnor fiber of an isolated
singularity is homotopy equivalent to a  bouquet of
$n$-spheres. 

These results combine to prove the vanishing of $\theta$ for $n$ even. See Theorem \ref{MainThm} and its Corollaries for 
the precise statement of our vanishing result.

Finally,  when $n$ is odd, \BvS prove that the $\theta$ pairing is induced by the linking form on the homology of the link of the singularity. Our Corollary \ref{Cor3}
may be interpreted as analogue of this result in algebraic $K$-theory; see the discussion at the end of Section \ref{sec:re}.

\section{On Sherman's star pairing} \label{sec:Sherman}
Suppose $Q$ is a regular ring, $f \in Q$ is a non-zero-divisor, and $(A,B)$ is a {\em matrix factorization} of $f$. Recall that this means $A$ and $B$ are $m
\times m$ matrices with entries in $Q$ such that $AB = fI_m = BA$. Note that $A$ may be regarded as an element of $GL_m(Q\adi{f})$ and hence it determines a
class $[A] \in K_1(Q\adi{f})$. The main result of this section, Corollary \ref{Cor2}, gives an explicit description of the image of $[A]$ under the boundary map
$$
K_1(Q\adi{f}) \xra{\del} G_0(Q/f)
$$
in the long exact localization sequence in $G$-theory. (Since $Q\adi{f}$ is regular, $K_1(Q\adi{f}) \cong G_1(Q\adi{f})$.) 
Our description of this image, and its proof, builds on  work of C. Sherman, which we review here.

\subsection{Pairings in $K$-theory}
We will need some results about various pairings in algebraic
$K$-theory, including Sherman's {\em  star pairing}.

For an exact category $\cP$, an object $M$ of it, and a pair of
commuting automorphisms $\alpha,
\beta$ of $M$, Sherman  \cite{ShermanConnII}
defines an element
$$
\alpha \sstar \beta \in K_2(\cP).
$$
In this definition, the group
$K_2(\cP)$ is taken to be $\pi_2 |G(\cP)|$ where $G(\cP)$ is a simplicial set, the 
``$G$-construction'', defined by Gillet-Grayson \cite{GG}. The reader is referred to the paper by Gillet and Grayson for the full definition, but let us recall
the definition of zero, one and two simplices.
A zero simplex in $G(\cP)$ is an ordered pair $(X,Y)$ of objects of $\cP$. An edge (i.e., a  one simplex)  connecting $(X,Y)$ to $(X',Y')$ is given by a pair of short exact
sequences of the form
$$
0 \to X \xra{i} X' \xra{p} Z \to 0
$$
and 
$$
0 \to Y \xra{j} Y' \xra{q} Z \to 0.
$$
(Notice the right-most non-zero object is the same in both sequences.) 
Equivalently, an edge is a pair of monomorphisms, together with a compatible collection of isomorphisms on the various representations of their cokernels.
We typically denote an edge as $(X,Y) \xra{(i,j)} (X',Y')$, leaving $Z$, $p$, and $q$ implicit.

A two simplex is represented by a pair of commutative diagrams
$$
\xymatrix{
X_0 \ar@{>->}[r] & X_1 \ar@{->>}[d] \ar@{>->}[r] & X_2
\ar@{->>}[d]&&
X'_0 \ar@{>->}[r] & X'_1 \ar@{->>}[d] \ar@{>->}[r] & X'_2
\ar@{->>}[d]
 \\
& X_{1/0}  \ar@{>->}[r]^i & X_{2/0} \ar@{->>}[d]^p 
&& & X_{1/0} \ar@{>->}[r]^i & X_{2/0} \ar@{->>}[d]^p \\
&& X_{2/1}  && && X_{2/1} \\
}
$$
such that 
$0 \to X_{i} \to X_{j} \to X_{i/j} \to 0$
and $0 \to X'_{i} \to X'_{j} \to X_{i/j} \to 0$, for all $0 \leq i < j \leq 2$, and
$0 \to X_{1/0} \to X_{2/0} \to X_{2/1} \to 0$ are short exact sequences. 

The element $\alpha \sstar \beta$ of $K_2(\cP)$ is represented by
the following diagram of simplices in $G(\cP)$:
\begin{equation} \label{E116}
\xymatrix{
& (M,M) \ar[rr]^{(1,\beta)}  \ar[dd]_{(\alpha,\alpha)}
\ar[rrdd]^{(\alpha, \alpha\beta)} && (M,M) \ar[dd]^{(\alpha,\alpha)} & \\
(0,0) \ar[ur] \ar[dr]  & & & & (0,0) \ar[ul] \ar[dl]\\
& (M,M) \ar[rr]^{(1,\beta)} & & (M,M). & \\
}
$$
In this diagram, each triangle is a two simplex. 
The lower-middle triangle is the two simplex
$$
\xymatrix{
M \ar@{>->}[r]^{\alpha} & M \ar@{->>}[d] \ar@{>->}[r]^{=} & M
\ar@{->>}[d]&&
M \ar@{>->}[r]^{\alpha} & M \ar@{->>}[d] \ar@{>->}[r]^{\beta} & M
\ar@{->>}[d]
 \\
& 0 \ar@{>->}[r] & 0 \ar@{->>}[d] 
&& & 0 \ar@{>->}[r] & 0 \ar@{->>}[d] \\
&& 0  && && 0, \\
}
\end{equation}
and the others are defined similarly.
Since the top and bottom paths in \eqref{E116} represent the same loop in $|G(\cP)|$,
this diagram represents a map from the two-sphere  to $|G(\cP)|$ and
hence an element of $K_2(\cP)$.

Sherman's pairing is functorial in the
following sense: If $F: \cP \to \cP'$ is an exact functor between
exact categories, $M \in \ob \cP$ and $\alpha, \beta$ are commuting
automorphisms of $M$, then $F_*: K_2(\cP) \to K_2(\cP')$ sends
$\alpha \sstar \beta$ to $F(\alpha) \sstar F(\beta)$.

If $\cP = \cP(S)$, the category of projective right
$S$-modules for some (not necessarily commutative) ring $S$, and $A, B
\in GL_n(S)$ are matrices that commute, we define
$$
A \sstar B \in K_2(S)
$$
by viewing $A, B$ as commuting automorphism of the right $S$-module
$S^n$, with $S^n$ thought of as column vectors and the action of $A, B$ given by 
multiplication on the left
of $S^n$. If $g:S \to S'$ is a ring map and $A, B$ are commuting
elements of $GL_n(S)$, then $g_*(A \sstar B) = g(A) \sstar g(B)$ holds
in $K_2(S')$.

Grayson \cite{GraysonAuto} has also defined a ``star'' pairing, defined for
pairs of commuting $n \times n$ invertible matrices $(A,B)$ in a (not necessarily
commutative) ring $S$. Grayson's definition amounts to 
$$
A \gstar B := D(A) \mstar D'(B)
$$
where  $D(A)$ and $D'(B)$ are the $3n \times 3n$ elementary matrices
$$
D(A) = 
\begin{bmatrix}
A & 0 & 0 \\
0 & A^{-1} & 0 \\
0 & 0 & I_n
\end{bmatrix} \quad \text{ and } \quad
D'(B) =
\begin{bmatrix}
B & 0 & 0 \\
0 & I_n & 0 \\
0 & 0 & B^{-1}
\end{bmatrix},
$$
and $\mstar$ denotes {\em Milnor's star pairing}. The latter pairing
was defined by Milnor \cite{MilnorKtheory} for
pairs of commuting {\em elementary} matrices $E_1$ and $E_2$ (in $E(R)
\subset GL(R) :=
GL_\infty(R)$) and is given by 
$$
E_1 \mstar E_2 = [\tilde{E}_1, \tilde{E}_2]
$$
where $\tilde{E}_1, \tilde{E}_2$ are lifts of $E_1, E_2$ to elements of
the Steinberg group, $St(R)$, under the canonical surjection $St(R) \onto E(R)$, and $[-,-]$ 
denotes
the commutator of a pair of elements of $St(R)$.

We also have the multiplication rule for $K_*(S)$, when $S$ is
a commutative ring, which gives the
pairing we will write as cup product: 
$$
- \cup - : K_1(S) \otimes_\Z K_1(S) \to K_2(S).
$$
There are several equivalent methods of describing this multiplication
rule \cite{WeibelProducts}; we use Milnor's \cite{MilnorKtheory} original description of it:
For $A \in GL_m(S)$ and $B \in GL_n(S)$, we write $[A], [B] \in K_1(R)$
for the associated $K$-theory classes. One then defines 
$$
[A] \cup [B] = D(A \otimes I_n) \mstar D'(I_m \otimes B)
$$
where $A \otimes I_n$ and $I_m \otimes B$ are identified with elements
of $GL_{mn}(R)$ by viewing each as an automorphism of 
$S^m \otimes_S
S^n$ and
choosing, arbitrarily, a basis of $S^m \otimes_S
S^n$. 

It is important to realize that $A \gstar B$
and $[A] \cup [B]$ are 
not the same
  element of $K_2(S)$  in general.  This is clear even for diagonal matrices: Say
  $A$ and $B$ are $2 \times 2$ diagonal matrices with diagonal entries
  $a_1,a_2$ and $b_1, b_2$ in a commutative ring $S$. Then $A \gstar B = [a_1] \cup [b_1] +
  [a_2] \cup [b_2] \in K_2(R)$ but $[A] \cup [B] = [a_1] \cup [b_1] +
  [a_1] \cup [b_2] + [a_2] \cup [b_1] + [a_2] \cup [b_2]$. 
(We have written the group law for $K_2(S)$ additively here.) 

We do have, however, the identity
$$
A \gstar uI_n
= [A] \cup [u]
$$
for any commutative ring $S$, unit $u \in S$ and matrix $A \in
GL_n(S)$.  In particular,
\begin{equation} \label{E89}
u \gstar v = [u] \cup [v]
\end{equation}
for any pair of units $u,v$ in a commutative ring.

It is clear that Grayson's star pairing is also
functorial for ring maps: If $g: S \to S'$ is any ring homomorphism and
$A$ and $B$ are elements of $GL_n(S)$ that commute, 
then
$$
g_*(A \gstar^S B) = g(A) \gstar^{S'} g(B)
$$
where $g_*: K_2(S) \to K_2(S')$ is the induced map on $K_2$ and the
superscript indicates in which ring the star operation is being
performed.

Both Grayson's and Sherman's operations are also preserved by Morita equivalence, in the
following sense: If we identify an $n \times n$ invertible matrix with
entries in $\Mat_m(S)$ with an $nm \times nm$ invertible matrix with
entries in $S$ in the usual manner, then
$$
\psi(A \sstar^{\Mat_m(S)} B) = A \sstar^{S} B
$$
and
$$
\psi(A \gstar^{\Mat_m(S)} B) = A \gstar^{S} B
$$
were $\psi: K_2(\Mat_m(S)) \map{\cong} K_2(S)$ is the canonical
isomorphism induced from the isomorphisms $\psi: St(\Mat_m(S))
\map{\cong} St(S)$ and $\psi: GL(\Mat_m(S)) \map{\cong} GL(S)$
(see \cite[4.5]{GraysonAuto}).

\subsection{Sherman's Theorem on the Connecting Homomorphism}
Sherman's Theorem \cite[3.6]{ShermanConnII} describes the
image of certain elements of the form $\alpha \sstar \beta \in K_1(\cP)$  under the boundary map in a 
long exact localization sequence.
To state it precisely, and for use again in Section \ref{sec:re}, we review some technical details of Sherman's work.

In an exact category $\cA$, a {\em mirror image sequence} is a pair $(E,F)$ of short exact sequences on the same three objects, but in the opposite order:
\begin{equation} \label{E115b}
E = (0 \to X \xra{i} Y \xra{p} Z \to 0) \and F = (0 \to Z \xra{j} Y \xra{q} X \to 0)
\end{equation}
Sherman \cite[p.18]{ShermanK1Exact} associates to a mirror image sequence $(E,F)$ a loop in $|G(\cA)|$ given by the diagram of
one simplices in $G(\cA)$
$$
(0,0) \to (X,X) \xra{\left(\iota_1,i\right)}  (X \oplus Z, Y) \xla{\left(\iota_2, j\right)}
(Z,Z) \leftarrow (0,0),
$$
where $\iota_1$ and  $\iota_2$ denote inclusions into the first and second summands,
and he writes $G(E,F)$ for the associated element of $K_1(\cA) = \pi_1 |G(\cA)|$.

Mainly for use in Section \ref{sec:re}, we recall here an alternative description of the class $G(E,F)$. A {\em double short exact sequence} in an exact category
$\cA$, defined originally by
Nenashev \cite{NenashevDSES},  is a pair of short exact sequences involving the same three objects, in the same order:
$$
l = (0 \to A \xra{i} B \xra{p} C \to 0,
0 \to A \xra{j} B \xra{q} C \to 0).
$$
Nenashev associates to a double short exact sequence $l$ the class $m(l) \in K_1(\cA)$ represented by the loop
$$
(0,0) \to (A,A) \xra{(i,j)} (B,B) \leftarrow (0,0).
$$
(Moreover, Nenashev \cite{NenashevDSES, NenashevK1GenRel} proves that $K_1(\cA)$ is generated by these loops and gives explicit generators for all the relations.)

Associated to a mirror image sequence \eqref{E115b}, we have the double short exact sequence 
$l_{(E,F)}$ given by  
\begin{equation} \label{E115c}
\begin{aligned}
&0 \to 
X \oplus Z \xra{\begin{bmatrix} 0 & j \\ 0 & 0 \\ 1 & 0 \end{bmatrix}}
Y \oplus Z \oplus X \xra{\begin{bmatrix} 0 & 1 & 0  \\ q & 0 & 0 \end{bmatrix}}
Z \oplus X \to 0 \\
&0 \to 
X \oplus Z \xra{\begin{bmatrix} i & 0 \\ 0 & 1 \\ 0 & 0 \end{bmatrix}}
Y \oplus Z \oplus X \xra{\begin{bmatrix} p & 0 & 0  \\ 0 & 0 & 1 \end{bmatrix}}
Z \oplus X \to 0.
\end{aligned}
\end{equation}

The following result of Sherman connects explicitly the two types of elements of $K_1(\cA)$.

\begin{prop}[Sherman] \cite[p. 164]{ShermanMirror} For any mirror image sequence \eqref{E115b} in an exact category $\cA$, we have
$$
G(E,F) = m(l_{E,F}) - G(Y, -1) \in K_1(\cA),
$$
where $(Y,-1)$ denotes the mirror image sequence
$$
(0 \to Y \xra{-1} Y \to 0 \to 0, 0 \to 0 \to Y \xra{1} Y \to 0).
$$
\end{prop}

Suppose now that $\cA$ is an abelian category, $\cS \subset \cA$ is a Serre
subcategory, and $\cP  = \cA/\cS$, the associated quotient abelian
category.  Recall that Quillen's Localization Sequence is a long
exact sequence of the form
$$
\cdots \to K_i(\cS) \to K_i(\cA) \to K_i(\cP) \xra{\partial} K_{i-1}(\cS) \to
\cdots.
$$
Suppose also that  $\alpha, \beta$ are commuting, injective
endomorphisms of an object $M$ of $\cA$ whose cokernels lie in $\cS$.
These form a commutative 
diagram 
$$
\xymatrix{
M \ar@{>->}[rr]^{\alpha} \ar@{>->}[dd]^{\beta} \ar@{>->}[ddrr]^{\beta\alpha}_{\alpha\beta}
&& M \ar@{>->}[dd]^{\beta} \\
\\
M \ar@{>->}[rr]^{\alpha} && M,  \\
}
$$
and by the ``snake chasing its tail Lemma'', this diagram determines the mirror image sequence $(E,F)$:
\begin{equation} \label{E116b}
\begin{aligned}
E & := (0 \to \coker(\alpha) \map{\beta} \coker(\beta\alpha) \to
\coker(\beta) \to 0) \\
F & := (0 \to \coker(\beta) \map{\alpha} \coker(\alpha\beta) \to
\coker(\alpha) \to 0).
\end{aligned}
\end{equation}
Since we are assuming that $\coker(\alpha)$ and $\coker(\beta)$ belong to $\cS$, 
the mirror image sequence $(E,F)$ determines the class $G(E,F) \in K_1(\cS)$. 
Also, 
the morphisms in $\cP = \cA/\cS$ induced by $\alpha$ and $\beta$, which we will write as
$\overline{\alpha}$ and  $\overline{\beta}$, are commuting automorphisms of $M \in \cP$, so that
we have the element  $\overline{\alpha} \sstar \overline{\beta} \in K_2(\cP)$.

\begin{thm}[Sherman]  \cite[3.6]{ShermanConnII}  \label{ShermansThm}
Under these hypotheses, 
\begin{equation} \label{ShermanE}
\partial(\overline{\alpha} \sstar \overline{\beta}) = 
\pm G(E,F) \in K_1(\cS)
\end{equation}
where $(E,F)$ is the mirror image sequence \eqref{E116b}.
\end{thm}

\subsection{A proof of (a weakened form of) a conjecture of Sherman}
Sherman has conjectured that his pairing $\sstar$ and Grayson's
pairing $\gstar$ coincide up to a sign. We prove here a slightly
weaker form of this conjecture:

\begin{thm} \label{SThm1}
If $S$ is a (not necessarily  commutative) ring and $A$ and $B$ are $n \times
  n$ invertible matrices with entries in $S$ that commute, then
$A \sstar B$ and $A \gstar B$ agree up to a sign and two-torsion --- i.e.,
$$
2 A \sstar B = \pm 2 A \gstar B \, \in K_2(S).
$$
\end{thm}

\begin{proof}
We first prove this when $S = \Z[x^{\pm 1},y^{\pm 1}]$ and $A$, $B$ are the $1 \times
1$ matrices $x$, $y$.
Since
$$
K_2(S) \cong K_2(\Z) \oplus K_1(\Z)^{\oplus 2}
\oplus K_0(\Z)
\cong \Z/2 \oplus (\Z/2)^{\oplus 2} \oplus \Z,
$$
to prove the theorem in this case, it suffices to show the images of 
$x \sstar y$ and $x \gstar y$ 
under the canonical map $K_2(S) \onto
\Z$, given by modding out torsion, are both generators of $\Z$. 

We have the localization long exact sequence
$$
\cdots \to K_2(\Z[x^{\pm 1}, y])\to K_2(S) \map{\partial} G_1(\Z[x^{\pm 1},y]/y)\map{0} 
K_1(\Z[x^{\pm 1}, y])\into K_1(S) 
\to
 \cdots.
$$
The map labelled as vanishing does so because the next map is
injective,  and this injectivity holds since 
$$
K_1(\Z[x^{\pm 1}]) \map{\cong}
K_1(\Z[x^{\pm 1}, y])
$$
is an isomorphism and the map
$\Z[x^{\pm 1}] \to S$ splits.

Observe $\Z[x^{\pm 1},y]/y \cong \Z[x^{\pm 1}]$ and hence 
$$
G_1(\Z[x^{\pm 1},y]/y) \cong
K_1(\Z[x^{\pm 1}]) \cong K_1(\Z) \oplus K_0(\Z) 
\cong \Z/2 \oplus \Z.
$$
Also, 
$$
K_2(\Z[x^{\pm 1}, y]) \cong K_2(\Z[x^{\pm 1}]) 
\cong K_2(\Z) \oplus K_1(\Z)
\cong \Z/2 \oplus \Z/2.
$$
From these calculations it follows that in order to prove the theorem
in this special case, it suffices to prove $\partial(x \sstar y)$ and
$\partial(x \gstar y)$
each map
to a generator of $\Z$ under the canonical surjection $G_1(\Z[x^{\pm
  1},y]/y) \onto \Z$ given by modding out torsion. 

To prove $\partial(x \sstar y)$ maps to a generator, 
we apply Sherman's equation \eqref{ShermanE}  with $\cA = \cM(\Z[x^{\pm1}, y])$ and $\cS$
the Serre subcategory of $y$-torsion modules, so that $\cA/\cS =
\cM(\Z[x^{\pm1}, y^{\pm1}])$. (In general, $\cM(A)$ denotes the abelian category of finitely generated $A$-modules.) 
This gives 
$\partial(x \sstar y) = \pm G(E,F)$ 
where $E, F$ are the short exact sequences in $\cS$
$$
E = \left(0 \to 0 \to \Z[x^{\pm 1}] \map{=} \Z[x^{\pm 1}] \to 0\right)
$$
and
$$
F = \left(0 \to \Z[x^{\pm1}] \map{x} \Z[x^{\pm1}] \to 0 \to 0\right).
$$
Under the isomorphism $K_1(\Z[x^{\pm 1}]) \cong
K_1(\cS)$, induced by the inclusion of $\cP(\Z[x^{\pm1},y]/y)$ into
$\cS$, the element $G(E,F)$ corresponds to $[x]
\in K_1(\Z[x^{\pm 1}])$, which maps to a generator under
$K_1(\Z[x^{\pm1}]) \onto \Z$.

To prove $\partial(x \gstar y)$ also maps to a generator of $\Z$, we first observe that $x
\gstar y = [x] \cup [y] \in K_2(S)$
by
\eqref{E89}. 
Note that $[x]$ lifts to an element of
$K_1(\Z[x^{\pm1}, y])$ and hence by \cite[1.1]{Suslin} we have
$$
\del([x] \cup [y]) = \overline{[x]} \cup \del'([y])
$$
where $\overline{[x]}$ is the image of $[x]$ under the map
$$
K_1(\Z[x^{\pm1}, y]) \to K_1(\Z[x^{\pm1}, y]/y) \cong K_1(\Z[x^{\pm 1}])
$$ 
and 
$$
\partial': K_1(S) \to G_0(\Z[x^{\pm 1},y]/y) \cong G_0(\Z[x^{\pm 1}]) \cong
K_0(\Z[x^{\pm 1}]) \cong \Z.
$$
is the boundary map in the localization long exact sequence. 
Now $\partial'([y]) = \coker(\Z[x^{\pm 1},y] \map{y}\Z[x^{\pm 1},y]) = 1$ in
$K_0(\Z[x^{\pm1}])$ and thus 
$$
\partial([x] \cup [y]) = \overline{[x]} \cup \partial'([y]) = 
\overline{[x]},
$$
whose image under 
$$
K_1(\Z[x^{\pm 1}]) \onto \Z
 $$
is a generator.
This completes the proof of the theorem in this special case.

The general form of the theorem now follows readily by naturality. Let $\Mat_n(S)$ denote
the ring of $n \times n$ matrices with entries in $S$ and define
$g: \Z[x^{\pm1},y^{\pm1}] \to \Mat_n(S)$ by $g(x) = A$ and $g(y) =
B$. This is well-defined since $A,B$ are assumed to commute.
We have proven $2 x \sstar y = \pm 2x \gstar y$. Using the naturality of
$\sstar$ and $\gstar$ for ring maps and Morita equivalence, it follows that
$$
\begin{aligned}
2 A \sstar B & = 2 \psi(A \sstar^{\Mat_n(S)} B) =
2 \psi g_*(x \sstar y)
=
\pm 2 \psi g_*(x \gstar y) \\
& =  
\pm 2 \psi(A \gstar^{\Mat_n(S)} B) =
\pm 2 A \gstar B. \qedhere 
\end{aligned}
$$
\end{proof}

\begin{cor} \label{Cor1}
Suppose $S$ is a commutative ring, $A,B$ are 
$n \times n$ invertible matrices with entries in $S$ and
  $AB = BA = fI_n$ for some (unit) $f \in S$. 
Then $2A \gstar B = 2 [A] \cup [f]$ and  
$2 A \sstar B = \pm 2 [A] \cup [f]$. 
\end{cor}

\begin{proof} By Theorem \ref{SThm1}, it suffices to prove the first equation.

Since $B = fA^{-1}$, we have $D'(B) =
  D'(fI_n)D'(A^{-1}) = D'(fI_n) D'(A)^{-1}$
and hence the bimultiplicativity of
  $\mstar$ (see \cite[8.1]{MilnorKtheory}) gives
$$
A \gstar B = D(A) \mstar D'(B) = D(A) \mstar D'(fI_n) + 
D(A) \mstar
D'(A^{-1})
=
[A] \cup [f] -
D(A) \mstar
D'(A)
$$
(where, as before, the group law for $K_2$ is written additively).

We now extend slightly an argument of Milnor \cite[8.2]{MilnorKtheory}. Let 
$$
P = 
\begin{bmatrix}
-I_n & 0 & 0 \\
0 & 0 & I_n \\
0 & I_n & 0 \\
\end{bmatrix}.
$$
Then $PD(A)P^{-1} = D'(A)$ and $PD(B)P^{-1} = D'(A)$. So, using the invariance of $\mstar$ 
under inner automorphisms
\cite[8.1]{MilnorKtheory}, we have
$$
D(A) \mstar D'(A)
= D'(A) \mstar D(A)
= - D(A) \mstar D'(A).
$$
It follows that $2D(A) \mstar D'(A) = 0$ and hence
$$
2 A \gstar B = 2 [A] \cup [f].
$$
\end{proof}

The following result will be used to reinterpret Hochster's $\theta$
invariant in the next section. We will apply it in the case when $Q$ is assumed to be regular, but state it here in its natural level of generality.

\begin{cor} \label{Cor2}
Let $Q$ be a commutative ring and  $f \in Q$ a non-zero-divisor, and set $R = Q/f$. 
Assume $(A,
B)$ is a matrix factorization of $f$ in $Q$; i.e., 
$A$ and $B$ are $m \times m$ matrices with entries in $Q$ such that $AB
= BA = fI_m$.

Then under the composition of 
$$
K_2(Q\ai{f}) \xra{\can}
G_2(Q\ai{f}) 
\map{\partial} G_1(R),
$$
where $\partial$ is the boundary map in the localization long
exact sequence, the image of the element $[A] \cup [f] \in K_2(Q\ai{f})$ is, up to a sign
and two torsion, equal to  $G(E,F) \in G_1(R)$, 
where $(E,F)$ is the mirror image sequence 
$$
\begin{aligned}
E & := (0 \to \coker(A) \map{\beta} R^m \to
\coker(B) \to 0) \\
F & := (0 \to \coker(B) \map{\alpha} R^m  \to
\coker(A) \to 0).
\end{aligned}
$$
\end{cor}

\begin{proof}
This follows immediately from Sherman's Theorem \ref{ShermansThm}
and 
Corollary \ref{Cor1}.
\end{proof}

\section{Reinterpreting Theta} \label{sec:re}
In this section, we use the results of the previous section to give  a reformulation of the 
$\theta$-pairing in terms of more familiar $K$-theoretic constructions.
Our results here are valid for a  general hypersurface ring.

Let $Q$ be a regular (Noetherian) ring, $\fm$ a maximal ideal of $Q$, and $f \in \fm$ a non-
zero-divisor.
Set $R = Q/f$ and 
also write $\fm$ for the image of $\fm$ in $R$.
Recall that Hochster's $\theta$ invariant may be viewed as a 
bilinear pairing of the form
$$
\theta = \theta_{(Q,\fm,f)}: G_0(\Spec R \smm) \times K_0(\Spec R \smm) \lra \Z,
$$
and is determined by the formula
$$
(M,N) \mapsto \len \Tor_{2j}^R(M,N) - \len \Tor_{2j+1}^R(M,N), \, j \gg
0,
$$
where $M$ and $N$ are finitely generated $R$-modules such that 
$pd_{R_{\fp}}N_\fp < \infty$ for all $\fp \ne \fm$.
We will need to know the $\theta$ pairing factors through localization at $\fm$; in fact, more 
is true:

\begin{prop} \label{prop814} 
Let $Q$, $\fm$, $f$, and $R$ be as above. Suppose $Q \to Q'$ is a flat ring map such that 
$\fm' := \fm Q'$ is a maximal ideal of $Q'$ and $Q/\fm \xra{\cong} Q'/\fm'$ is an isomoprhism. 
Let $f' \in Q'$ be the image of $f$ in $Q'$ and set $R' = Q'/f' = Q' \otimes_Q R$.
The triangle
$$
\xymatrix{
G_0(\Spec(R) \smm) \times K_0(\Spec(R) \smm) \ar[rd]^{\phantom{XX} \theta^R} \ar[dd]^{\phi^* \times 
\phi^*} \\
& \Z \\
G_0(\Spec(R' \sm \fm')) \times K_0(\Spec(R') \sm \fm') \ar[ru]^{\theta^{R'}} \\
}
$$
commutes, where $\phi: \Spec(R') \sm \fm' \to \Spec(R) \sm \fm$ is the induced flat 
morphism.

In particular, the $\theta$ pairing factors through the localization at $\fm$; that is,
$$
\xymatrix{
G_0(\Spec(R) \smm) \times K_0(\Spec(R) \smm) \ar[rd]^{\phantom{XX} \theta^R} \ar[dd] \\
& \Z \\
G_0(\Spec(R_\fm) \smm) \times K_0(\Spec(R_\fm) \smm) \ar[ru]^{\theta^{R_\fm}} \\
}
$$
commutes.
\end{prop}

\begin{proof} The hypothesis $\fm Q' = \fm'$ implies  that $Q \to Q'$ does indeed determine 
a flat morphism $\phi: \Spec(R) \sm \fm \to \Spec(R') \sm \fm'$ of
  schemes.
Since $R \to R'$ is flat, we have an isomorphism
$$
\Tor^R_j(M,N) \otimes_R R' \cong \Tor^{R'}_j(M \otimes_R R', N \otimes_R R').
$$
Finally, for any finitely generated $R$-module $T$ supported at $\fm$, we have
$$
\len_R(T) = \len_{R'}(T \otimes_R R')
$$
since $R' \otimes_R R/\fm \cong R'/\fm R' = R'/\fm'$ by assumption.
\end{proof}

For any Noetherian scheme $U$, we write $K_*(U)$ for the (Quillen) $K$-groups of the exact 
category of locally free coherent sheaves on $U$ and $G_*(U)$ for the
$K$-groups of the abelian category of all coherent sheaves.
The groups $K_*(U)$ form a graded ring
with the multiplication rule, which we write as
cup product $- \cup -$, induced by tensor product. Tensor product also
defines the cap product pairing
$$
- \cap -: K_i(U) \times G_j(U) \to G_{i +j}(U)
$$
making $G_*(U)$ into a graded $K_*(U)$-module.

For any integer $\ell \geq 0$, let $K_*(U, \Z/\ell)$ and $G_*(U, \Z/\ell)$ denote
$K$-theory and $G$-theory with $\Z/\ell$
coefficients. These are  defined as the homotopy groups of 
spectra
obtains from the $\cK$- and
$\cG$-theory spectra by smashing them with the mod-$\ell$ Moore space.
There are long exact coefficient sequences
$$
\cdots \to K_m(U) \xra{\cdot \ell} K_m(U) \to K_m(U, \Z/\ell) \to
K_{m-1}(U) \to \cdots
$$
and similarly for $G$-theory.

We also will need the
mod-$\ell$ versions of the cup and cap product pairings. 
To avoid complications in the multiplication rule for
the Moore spaces for small primes, we assume for simplicity that
$\Z/\ell$ has no $2$- or $3$-torsion; i.e., 
either $\ell = 0$ or it is not divisible
by $2$ or $3$. In this case, $K_*(U, \Z/\ell)$ is a graded
ring under cup product 
and $G_*(U, \Z/\ell)$ is a module over this ring under cap product.
Moreover, $K_*(U) \to K_*(U, \Z/\ell)$ is a ring map and $G_*(U) \to
G_*(U, \Z/\ell)$ is $K_*(U)$-linear. 
See \cite[A.6]{Thomason} for more information.

The open complement of the closed subscheme $\Spec(R) \smm$ of
$\Spec(Q) \smm$ is $\Spec Q \ai{f}$ and (using that $Q$ is regular) we
have a 
long exact localization sequence in $K$-theory:
\begin{equation} \label{E813}
\begin{aligned}
\cdots \to & K_{i+1}(\Spec(Q)  \smm, \Z/\ell) \to K_{i+1}(\Spec Q\ai{f}, \Z/\ell) \map{\del} 
G_i(\Spec(R)
\smm, \Z/\ell) \\
& \to  K_{i}(\Spec(Q)  \smm, \Z/\ell) \to  \cdots. \\
\end{aligned}
\end{equation}
We also have the long exact localization
sequence in $G$-theory associated to the closed subscheme $\Spec(R/\fm)$ of $\Spec(R)$, 
\begin{equation} \label{E813b}
\cdots \to G_{i+1}(\Spec R, \Z/\ell) \to G_{i+1}(\Spec(R) \smm, \Z/\ell) \xra{\del'} 
K_i(R/\fm, \Z/\ell)
\to G_i(R, \Z/\ell) \to \cdots.
\end{equation}

\begin{defn} \label{tthetadef}
Assume $Q$ is a (not necessarily regular) Noetherian ring, $\fm$ is a maximal
  ideal, $f \in \fm$ is any element, and $\ell$ is an integer such that
  $\Z/\ell$ has no $2$- or $3$-torsion. Set $R = Q/f$. 
Let $\partial$ and $\partial'$ be the boundary maps in the $G$-theory localization sequences 
\eqref{E813} and \eqref{E813b}.

Define $\ttheta = \ttheta_{(Q,f,\fm), \ell, (i,j)}$ to be the pairing 
$$
\ttheta: G_i(Q\ad{f}, \Z/\ell) \times K_j(\Spec(R)
\setminus \fm, \Z/\ell) \to K_{i+j-1}(R/\fm, \Z/\ell)
$$
defined by the formula
$$
\ttheta(\alpha, \gamma) = \partial'\left(\partial(\alpha \cap [f]) \cap \gamma\right),
$$
where $[f]  \in K_1(Q \ad{f})$ denotes the class of the unit $f$ of $Q\ai{f}$. 
\end{defn}

Taking $i = 1$ and $ j = 0$, we have in particular the pairing
\begin{equation} \label{E32}
\ttheta  : G_1(Q\ad{f}, \Z/\ell) \times K_0(\Spec(R)
\setminus \fm, \Z/\ell) \to K_0(R/\fm, \Z/\ell) \cong \Z/\ell,
\end{equation}
where the isomorphism sends $[R/\fm]$ to $1 \in \Z/\ell$.
The goal of this section is to relate this pairing to Hochster's
$\theta$ pairing.

We shall need the following analogue of Proposition \ref{prop814} for
$\ttheta$:

\begin{prop} \label{prop815b} Suppose $(Q, \fm, f, \ell, R)$ are as in Definition \ref{tthetadef} 
and that
$Q \to Q'$ is a flat ring map such that 
$\fm' := \fm Q'$ is a maximal ideal of $Q'$ and $Q/\fm \xra{\cong} Q'/\fm'$ is an isomoprhism. 
Let $f' \in Q'$ be the image of $f$ in $Q'$ and set $R' = Q'/f' = Q' \otimes_Q R$.
Then the square
$$
\xymatrix{
G_i(Q\adi{f}, \Z/\ell) \times K_j(\Spec(R) \sm \fm, \Z/\ell) \ar[r]^{\phantom{XXXXX} \ttheta}  
\ar[d] & K_{i+1-1}(R/\fm, \Z/\ell) \ar[d]^{\cong} \\
G_i(Q'\adi{f'}, \Z/\ell) \times K_j(\Spec(R') \sm \fm', \Z/\ell) \ar[r]^{\phantom{XXXXX} \ttheta}  & 
K_{i+1-1}(R'/\fm', \Z/\ell) \\
}
$$
commutes.
\end{prop}

\begin{proof} Write $\phi: \Spec(Q') \to \Spec(Q)$ for the associated map of affine schemes.
The hypotheses ensure that each square in 
$$
\xymatrix{
\Spec(R') \sm \fm' \ar[r] \ar[d]^{\phi} & \Spec(Q') \sm \fm' \ar[d]^{\phi} \ar[r] & \Spec(Q') 
\ar[d]^{\phi} & \Spec(Q'\adi{f'}) \ar[l] \ar[d]^{\phi} \\
\Spec(R) \sm \fm \ar[r]  & \Spec(Q) \sm \fm \ar[r] & \Spec(Q) & \Spec(Q\adi{f}) \ar[l]  \\
}
$$
a pull-back. It follows that 
the diagram
$$
\xymatrix{
\cG(\Spec(R) \sm \fm) \ar[r] \ar[d]^{\phi^*} & \cG(\Spec(Q) \sm \fm) \ar[r] \ar[d]^{\phi^*}  & 
\cG(\Spec(Q\adi{f})) \ar[d]^{\phi^*}  \\
\cG(\Spec(R') \sm \fm) \ar[r]  & \cG(\Spec(Q') \sm \fm') \ar[r]  & \cG(\Spec(Q'\adi{f'}))  \\
}
$$
is homotopy commutative. This gives us that $\phi^*$ commutes with the boundary maps 
$\partial$ in the associated long exact sequence. For any flat map $\phi$,
$\phi^*$ commutes with $\cap$ product. This proves that
$$
\phi^*\left(\partial(\a \cap [f]) \cap \gamma\right) = 
\partial(\phi^*(\a) \cap [f']) \cap \phi^*(\gamma) \in G_{i+j}(\Spec(R') \sm \fm').
$$
Similarly, the diagram
$$
\xymatrix{
\cG(\Spec(R/m)) \ar[r] \ar[d]^{\phi^*} & \cG(R) \ar[d]^{\phi^*} \ar[r] & \cG(\Spec(R) \sm \fm) 
\ar[d]^{\phi^*} \\
\cG(\Spec(R'/m')) \ar[r] & \cG(R') \ar[r] & \cG(\Spec(R') \sm \fm') \\
}
$$
is homotopy commutative so that $\phi^*$ commutes with the boundary maps $\partial'$.
\end{proof}

\begin{ex} We will use the previous result, in particular, when $Q' = Q^\hen_\fm$, the Henselization of 
$Q$ at the maximal ideal $\fm$. 
That is, the $\ttheta$ pairing for $R = Q/f$ factors through its
  Henselization at $\fm$. 
\end{ex}

Assume $Q$ is a regular ring and $f \in Q$ is a non-zero-divisor.
Given a matrix factorization $(A,B)$ of $f$ in $Q$, the module $\coker(A)$ is
annihilated by $f$ and thus may be regarded as an $R$-module, where $R := Q/f$. It is
necessarily a
MCM $R$-module, as seen by the depth formula and the fact that
$pd_Q(M) = 1$. Moreover, if $Q$ is local, every MCM $R$-module is the
cokernel of some matrix factorization.
Recall that $A$ is an
  invertible matrix when regarded as a matrix with entries
  in $Q\ai{f}$ and we write $[A]$ for the class in $K_1(Q\ai{f})$ it determines. We 
  write $[A]_\ell$ for image of this class in
  $K_1(Q\ai{f}, \Z/ \ell)$. (Note that $[A]_0 = [A]$.)

For finitely generated $R$-modules $M$ and  $N$ 
 such that $N_\fp$ has finite
  projective dimension for all $\fp \ne \fm$,
let $\theta_\ell(M,N) \in \Z/\ell$ be the value of $\theta(M,N)$
modulo $\ell$. 

Recall that such a module $N$ determines a coherent sheaf on $\Spec(R) \smm$ that admits a finite resolution by locally free coherent sheaves, and hence $N$
determines a class $[N'] \in  K_0(\Spec(R)
  \smm)$. We write $[N']_\ell$ for the image of $[N']$ in $K_0(\Spec(R)
  \smm, \Z/\ell)$. (As before, $[N]_0 = [N]$.) 

\begin{thm} \label{Thm1}
Assume $Q$ is a regular ring, $\fm$ is a maximal ideal of $Q$ and $f \in \fm$ is a
non-zero-divisor. Let $R = Q/f$ and let $\fm$ also denote the image of $\fm$ in
$R$. Let $(A,B)$ be a matrix factorization of $f$ in $Q$  and let 
$M = \coker(A)$ be the associated 
 MCM $R$-module. Let
  $N$ be a finitely generated $R$-module such that $pd_{R_\fp}(N_\fp) < \infty$  
for all $\fp \ne \fm$, and let $[N']$ be the class it determines in $K_0(\Spec(R) \smm)$. 

For any integer  $\ell$ such that $\Z/\ell$ has no $2$- or $3$-torsion,
$$
\theta_\ell(M,N) = \pm  \ttheta_\ell([A]_\ell, [N']_\ell)
$$
\end{thm}

\begin{proof} The compatibility of cup and cap products with the
  canonical maps $K_*(U) \to K_*(U, \Z/\ell)$ implies that the diagram
$$
\xymatrix{
K_1(Q\ai{f}) \times K_0(\Spec(R) \smm) \ar[r]^{\phantom{XXXX} \ttheta} \ar[d] & G_0(R/\fm) 
\ar[d] \\
K_1(Q\ai{f}, \Z/\ell) \times K_0(\Spec(R) \smm , \Z/\ell)
\ar[r]^{\phantom{XXXXXXX} \ttheta_\ell} & G_0(R/\fm , \Z/\ell)  
}
$$
commutes. It therefore suffices to prove the theorem with integral coefficients.

To shorten notation, set $U = \Spec(Q)$, $U' = U \smm$, $V = \Spec(R)$ and $V' = V \smm$. Note that $U'$ is open in $U$, $V$ is closed in $U$ and $V' = V
\cap U$. 

Let $N'$ denote the coherent sheaf on $V'$ given by the restriction of $N$, viewed as a coherent sheaf on $V$,  to $V'$.
Since a high enough syzygy for an $R$-free resolution of $N$ will be
locally free on $V'$  and since the equation in question
depends only on the class of $N'$ in $K_0(V')$, we may
assume $N'$ is a locally free coherent sheaf on $V'$.

From Corollary \ref{Cor2} we have 
$$
2 \del([A] \cup [f]) = \pm 2 G(E,F)  \in  G_1(\Spec R) = G_1(V).
$$ 
where $E$ and $F$ are the short exact sequences of coherent sheaves on $V = \Spec(R)$ 
$$
E := (0 \to \coker(A) \to \cO_V^m \to \coker(B) \to 0)
$$
and
$$
F := (0 \to \coker(B) \to \cO_V^m \to \coker(A) \to 0). 
$$
(We are assuming $A$ and $B$ are $m \times m$ matrices.) 
The image of $G(E,F)$ under $G_1(V) \to G_1(V')$ is $G(E', F')$, where 
$$
E'  := \left( 0 \to \coker(A') \to \coker \cO_{V'}^m \to \coker(B')
  \to 0 \right)
$$
and
$$
F' := \left( 0 \to \coker(B') \to \coker \cO_{V'}^m \to \coker(A')
  \to 0\right),
$$
where $A'$ and $B'$ be the injective maps $\cO_U^m \into \cO_U^m$ 
induced by $A$ and $B$. 
Thus
$$
2 \del([A] \cup [f]) = \pm 2 G(E',F')  \in  G_1(V').
$$

We may describe the mirror image sequence $(E, F)$ in the following equivalent way.
Let $\overline{A}, \overline{B}$ be the matrices in $R = Q/f$ determined by $A,B$.
We have an unbounded, two-periodic exact sequence of locally free $\cO_V$-modules
\begin{equation} \label{E115}
\cdots \to \cO_{V}^r \map{\overline{A}} \cO_{V}^r \map{\overline{B}}
 \cO_{V}^r \map{\overline{A}} \cO_{V}^r \map{\overline{B}}  \cdots.
\end{equation}
Then $E$ and $F$ are the two canonical  short-exact sequences involving sysygies coming from this two periodic exact sequence.

The key result we need is that
$$
\del'(G(E',F') \cap [N']) = \theta(M,N),
$$
where $\del'$ is the boundary  map $G_1(V') \to G_0(R/\fm) \cong \Z$.
This is essentially proven by
Buchweitz and van Straten \cite[3.4]{BvS}, but they use the notation of double short exact sequences, and not that of mirror image sequences. We proceed to summarize their proof
in the notation we need.

Since $N'$ is locally free on $V'$, the exact sequences $E'$ and $F'$ remain exact upon tensoring with
$N'$ and it follows that 
$$
G(E', F') \cap [N'] = G(E'  \otimes_{\cO_{V'}} N', F' \otimes_{\cO_{V'}} N') \in G_1(V').
$$
Tensoring the sequence \eqref{E115} with $N$ yields the two periodic complex of coherent sheaves on $V$, 
$$
\cdots \xra{\overline{B} \otimes \id_N}  \cO_V^r \otimes_{\cO_V} N \xra{\overline{A} \otimes \id_N} 
\cO_V^r \otimes_{\cO_V} N \xra{\overline{B} \otimes \id_N} 
\cO_V^r \otimes_{\cO_V} N \xra{\overline{A} \otimes \id_N} 
\cdots
$$
and, again using that $N'$ is locally free on $V'$, this complex  is exact on the open subset $V'$ of $V$. From this complex, one builds what \BvS call a
``cyclic diagram'':
$$
\zeta := 
\left(\xymatrix{
0 \ar[r] &  \ker(\overline{A} \otimes \id_N) \ar[r] &  \cO_V^r \ar[r] \ar@{=}[d] &  \im(\overline{A} \otimes \id_N) \ar[r] \ar[d]^{\subseteq} &  0 \\
0        &  \im(\overline{B} \otimes \id_N) \ar[l] \ar[u]^\subseteq &  \cO_V^r \ar[l] &  \ker(\overline{B} \otimes \id_N) \ar[l] & \ar[l]  0 \\
}\right)
$$
(Note that the digram does {\em not} commute.) The left and right vertical arrows in this are the canonical inclusions.
\BvS define $\{\zeta\} \in G_1(R) = G_1(V)$ to be $m(l_\zeta)$ where $l_\zeta$  a certain double short exact sequence associated to a cyclic diagram $\zeta$; see
\cite[p. 247]{BvS} for the precise formula.

The image of $\{\zeta\}$ under $G_1(V) \to G_1(V')$ is $\{\zeta'\}$ where
$$
\zeta' := 
\left(\xymatrix{
0 \ar[r] &  \ker(\overline{A}' \otimes \id_{N'}) \ar[r] &  \cO_{V'}^r \ar[r] \ar@{=}[d] &  \im(\overline{A}' \otimes \id_{N'}) \ar[r] \ar@{=}[d] &  0 \\
0        &  \im(\overline{B}' \otimes \id_{N'}) \ar[l] \ar@{=}[u] &  \cO_{V'}^r \ar[l] &  \ker(\overline{B}' \otimes \id_{N'}) \ar[l] & \ar[l]  0 
}\right)
$$
Observe that the vertical arrows are now all identity maps, and thus we may interpret $\zeta'$ as determining the  mirror image sequence 
$(E'',F'')$ where
$$
E'' := (0 \to \ker(\overline{A}' \otimes \id_{N'}) \xra{\subseteq} \cO_{V'}^r \xra{\overline{A'} \otimes \id_{N'}} \im(\overline{A}' \otimes \id_{N'}) \to 0)
$$
and
$$
F'' := (0 \to \im(\overline{A}' \otimes \id_{N'}) \xra{\subseteq} \cO_{V'}^r \xra{\overline{B'} \otimes \id_{N'}} \ker(\overline{A}' \otimes \id_{N'}) \to 0).
$$
Since $\ker(\overline{A}') = \im(\overline{B}') \cong \coker(\overline{A}')$
and $\ker(\overline{B}') = \im(\overline{A}') \cong \coker(\overline{B}')$
we have an isomorphism of mirror image sequences
$$
(E'', F'') \cong (E'  \otimes_{\cO_{V'}} N', F' \otimes_{\cO_{V'}} N').
$$
Moreover, by comparing the definitions $l_{\zeta'}$ and $l_{(E'',F'')}$ 
(see \cite[p. 247]{BvS} and \eqref{E115c}), it is clear that 
$$
l_{\zeta'} = - l_{(E'',F'')}.
$$ 
  
Using also that $m(l_{(E'',F'')}) = G(E'', F'') + G(\cO_{V'}^r,-1)$ \cite[p. 164]{ShermanMirror} we obtain 
$$
-\{\zeta'\} = (G(E',F') \cap [N'])  + G(\cO_{V'}^r, -1) \in G_1(V').
$$
The class $G(\cO_{V'}^r, -1)$ clearly lifts along $G_1(V) \to G_1(V')$ to the class $G(\cO_V^r, -1)$ and hence $\del'(G(\cO_{V'}^r, -1)) = 0$ and
$$
\del'(G(E',F') \cap [N'])  = \del'(\zeta').
$$
Finally,
\BvS show \cite[Proof of 3.4]{BvS} that
$$
\del'(\{\zeta'\}) = \theta(M,N).
$$

We have proven that 
$$
2 \theta(M,N) = \pm 2  \del'\left(\del([A] \cup [f]) \cup
  [N']\right) = \pm 2 \ttheta_\ell(M,N)
$$
and since this is an equation in $\Z$, we may divide by $2$.
\end{proof}

\begin{cor} \label{Cor3}
Let $Q$ be a regular ring, $\fm$ a maximal ideal of $Q$, $f \in \fm$ a 
  non-zero-divisor, $R = Q/f$, and $\ell \geq 0$ an integer such that
  $\Z/\ell$ is without $2$- and $3$-torsion.
The diagram
$$
\xymatrix{
K_1(Q\ad{f}, \Z/\ell) \times K_0(\Spec(R) \setminus \fm , \Z/\ell) \ar[d]^{\del
  \times id} \ar[rd]^{\phantom{XXXXX} \ttheta} \\
G_0(\Spec(R) \setminus \fm , \Z/\ell) \times K_0(\Spec(R) \setminus
\fm , \Z/\ell) \ar[r]_{\phantom{XXXXXXXXXXXXXX} \theta} & \Z/\ell
}
$$
commutes up to a sign.
\end{cor}

\begin{proof} By Propositions \ref{prop814} and \ref{prop815b}, $\theta$ and $\ttheta$ 
factor
through the localization of $R$ at $\fm$, and hence 
we may  assume $Q$ is local.

Theorem \ref{Thm1} shows that $\ttheta_\ell(A,N) = \theta_\ell(\coker(A),
  N)$ for all matrix factorizations $(A,B)$ of $f$ in $Q$ and all $N$.
  An arbitrary class of $K_1(Q\adi{f})$ is determined by an invertible matrix $T$ with entries in $Q\adi{f}$ (not necessarily coming from a matrix
  factorization). 
In $G_0(R)$, we may
  represent $\partial([T])$ as the difference of the classes of MCM
  modules. Since $Q$ is assumed local, every MCM
  $R$-module is represented as the cokernel of a matrix
  factorization. It follows that $K_1(Q\ad{f})$ is generated by
  classes coming from matrix factorizations of $f$ and the image of
  $K_1(Q) \to K_1(A \ad{f})$. It therefore suffices to prove $\ttheta$
  annihilates the image of $K_1(Q)$. 

Let then $T$ be an invertible matrix with entries in $Q$.
Using \cite[1.1]{Suslin}, we have 
$$
\ttheta([T], N)  = \del'(\partial([T] \cup [f]) \cap [N])  = \del'((\overline{[T]} \cap \partial [f]) \cap [N])
$$
where $\overline{[T]}$ denotes the image of $[T]$ under the canonical map $K_1(Q) \to K_1(\Spec(R) \sm \fm)$.  The map $\partial$ factors as
$$
K_1(Q\adi{f}) \xra{\tilde{\partial}} G_0(R) \xra{\can} G_0(\Spec(R) \smm). 
$$
Since $\tilde{\partial}([f]) = [R]$ we have $\overline{[T]} \cap \partial [f] = \overline{[T]}$ and hence
$$
\ttheta([T], N)  =  \del'(\overline{[T]} \cap [N]).
$$
Finally, since  $\overline{[T]}$ and $[N]$ lift to elements of $K_1(R)$
and $G_0(R)$, respectively, $\overline{[T]} \cap [N] \in G_1(\Spec(R) \smm)$ lifts to an element of $G_1(R)$. It follows that
$\del'(\overline{[T]} \cap [N]) = 0$.
\end{proof}

Let us indicate how the previous result relates to the work of Buchweitz and van Straten, and least on  the level of analogy. We start with:

\begin{prop} Assume $(Q, \fm, k)$ is a Henselian, regular local ring with algebraically closed residue field $k$ and
$0 \ne f \in \fm$. Set $R = Q/f$. 
Assume 
$\ell$ is a  positive integer that is relatively prime to $\chr(k)$ and is such that 
$[R/\fm] = 0$ in $G_0(R, \Z/\ell)$.

Then the boundary map
\begin{equation} \label{E1219}
K_1(Q\ad{f}, \Z/\ell) \xra{\cong} G_0(\Spec(R) \setminus \fm, \Z/\ell)
\end{equation}
is an isomorphism. 
\end{prop}

\begin{rem} Before proving this, let us show $[R/\fm] = 0$  in $G_0(R, \Z/\ell)$, at least when $\dm(R) > 0$,  for all but a finite number of primes $\ell$, so that the
  hypotheses of the Proposition are often met. 

Let $\fp$ be a prime in $R$ of height $\dm(R) - 1$ (recall that we assume $\dm(R) > 0$) and choose $g \in \fm \setminus \fp$. Then the short exact
  sequence
$$
0 \to R/\fp \xra{g} R/\fp \to R/(\fp, g) \to 0
$$
gives $[R/(\fp, g)] = 0$ in $G_0(R)$. But $R/(\fp, g)$ is a finite length $R$-module, say of length $N$, and hence 
$[R/(\fp, g)] = N [R/\fm]$. 
This proves $N [R/\fm] = 0$ in $G_0(R)$ and thus if $\ell$ is a prime such that $\ell \nmid N$, then $[R/\fm]$ is divisible by $\ell$ in $G_0(R)$.
\end{rem}

\begin{proof} 
The localization sequences for $G$-theory (with $\Z/\ell$ coefficients), and their naturality,  applied to the following schemes, closed subschemes, and open complements
$$
\Spec(R/\fm) \subseteq \Spec(R) \supseteq \Spec(R) \setminus \fm,
$$
$$
\Spec(Q/\fm) \subseteq \Spec(Q) \supseteq \Spec(Q) \setminus \fm, \text{ and}
$$
$$
\Spec(R) \setminus \fm \subseteq \Spec(Q) \setminus \fm \supseteq \Spec(Q\adi{f})
$$
yield the commutative diagram
$$
\xymatrix{
G_1(R) \ar[d] \ar[r] & G_1(R \setminus \fm)  \ar[r] \ar@{->>}[d] & G_0(R/\fm) \ar[r]^{0} \ar[d]^{\cong} & G_0(R) \ar[d] \\
G_1(Q) \ar[r]^0 & G_1(Q \setminus \fm) \ar[r]^{\cong} \ar[d]^0 & G_0(Q/\fm) \ar[r] & G_0(Q) \\
& G_1(Q\adi{f}) \ar[d] \\
& G_0(R \setminus \fm)  \ar[d]^0\\
& G_0(Q \setminus \fm) \ar[d]^{\cong} \\
& G_0(Q\adi{f}) \\
}
$$
with exact rows and column.
(Here, we have suppressed the coefficients of $\Z/\ell$ and written 
$R \setminus \fm$ and 
$Q \setminus \fm$ for the non-affine schemes 
$\Spec(R) \setminus \fm$ and $\Spec(Q) \setminus \fm)$.) The Proposition follows immediately from this diagram, provided we justify the labels.

The hypothesis that $[R/\fm] = 0$ in $G_0(R, \Z/\ell)$ gives as that  
$G_0(R/\fm, \Z/\ell) \to G_0(R, \Z/\ell)$ is the zero map, as labelled. The map
$G_1(Q, \Z/\ell) \to G_1(Q \setminus \fm, \Z/\ell)$ is the zero map since
$$
G_1(Q, \Z/\ell) \cong K_1(Q, \Z/\ell) \cong K_1(Q/\fm, \Z/\ell) = 0
$$ 
by \cite{GabberHenelian} and \cite{Suslin}.
A diagram chase then gives that $G_1(R \setminus \fm, \Z/\ell) \to G_1(Q \setminus \fm, \Z/\ell)$ is onto and hence that
$G_1(Q \setminus \fm, \Z/\ell) \to G_1(Q\adi{f}, \Z/\ell)$ is the zero map. 
The group $G_0(Q \setminus \fm, \Z/\ell) \cong K_0(Q \setminus \fm, \Z/\ell)$ is isomorphic to $\Z/\ell$ and is generated by the class of the structure sheaf,
since $G_0(Q, \Z/\ell) \cong K_0(Q, \Z/\ell) \cong \Z/\ell$ (using that  $Q$ is local) and 
$K_0(Q, \Z/\ell)  \to K_0(Q \setminus \fm, \Z/\ell)$ is onto.
It follows that $K_0(Q \setminus \fm, \Z/\ell) \to K_0(Q\adi{f}, \Z/\ell)$ is an isomorphism and hence that
$G_0(R \setminus \fm, \Z/\ell) \to G_0(Q \setminus \fm, \Z/\ell)$ is the zero map.
\end{proof}

Assuming the hypotheses of the Proposition are met and also that $\ell$ is not divisible by $2$ or $3$, then the mod-$\ell$ theta pairing, 
$$
\theta_\ell: G_0(\Spec(R) \setminus \fm, \Z/\ell) \times
G_0(\Spec(R) \setminus \fm, \Z/\ell) \to \Z/\ell,
$$
may be given by the formula 
\begin{equation} \label{E17}
\theta_\ell(\a, \b) = \langle \gamma(\a), \b \rangle.
\end{equation}
Here, we define
\begin{equation} \label{E116d}
\gamma: G_0(\Spec(R) \setminus \fm, \Z/\ell) \to G_1(\Spec(R) \setminus \fm, \Z/\ell)
\end{equation}
as the composition of the inverse of the 
isomorphism \eqref{E1219} and the map 
\begin{equation} \label{E116c}
K_1(Q\ad{f}, \Z/\ell) \xra{\partial(- \cup [t])} G_1(\Spec(R) \setminus \fm, \Z/\ell).
\end{equation}
(In the next section, we note that the latter map is the ``specialization map'' in $K$-theory, and we will explore its properties.)
The pairing
$$
\langle -,- \rangle: G_1(\Spec(R) \setminus \fm, \Z/\ell) \times G_0(\Spec(R) \setminus \fm, \Z/\ell) \to \Z
$$
is the composition of cup product, the boundary map 
$G_1(\Spec(R) \setminus \fm, \Z/\ell) \xra{\partial'} G_0(R/\fm, \Z/\ell)$ and the canonical isomorphism
$G_0(R/\fm) \cong \Z/\ell$. 

The isomorphism \eqref{E1219} is the analogue of the Alexander duality isomorphism $H^{2m+1}(S \setminus L) \xra{\cong} H^{2m}(L)$ and the map \eqref{E116c}
is the analogue of the map 
$\rho_t^*: H^{j+1}(S \setminus L) \to H^{j+1}(L)$ induced by the ``push aside'' map $\rho_t: L \to \partial_t F \subseteq (S \setminus L)$.
Thus the map $\gamma$ defined in \eqref{E116d} is the analogue of the map 
$$
\gamma^{\top}: H^{j}(L) \to H^{j+1}(L) 
$$
occurring in the work of \BvSns. They prove \cite[5.1]{BvS}
$$
\link(\a, \b) = \langle \gamma^\top(\a) , \b \rangle
$$
for classes $\a, \b$ in $H^{n-1}(L) \cong H_n(L)$, where $\link$ is the linking form on $H_n(L)$. (The linking form may be defined as the restriction of the
Seifert form, defined
on the homology of the Milnor fiber $H_n(F_t)$, along the canonical map $H_n(\partial(F_t)) \to H_n(F_t)$, using the diffeomorphism $L \cong \partial F_t$.)

We may thus interpret Corollary \ref{Cor3} as saying that the mod-$\ell$ theta pairing is the analogue for algebraic $K$-theory with finite coefficients 
of the linking form in topology.

\section{Specialization with finite coefficients (Following Suslin)}

In the previous section we have given a new interpretation of the $\theta$ pairing in terms of 
the pairing $\ttheta$ defined as
$$
\ttheta(\a, \g) = \partial'(\partial(\a \cap [f]) \cap \g).
$$
Under some additional assumptions, a portion of this formula, namely the map 
$$
\sigma: G_i(Q\adi{f}, \Z/\ell) \to G_i(\Spec(R) \sm \fm, \Z/\ell)
$$
sending $\a$ to $\partial(\a \cap [f])$, can be interpreted as a {\em specialization map}. The 
goal of this section is to make this precise and to establish
properties of this specialization map. When working with the algebraic analogue of the Milnor fiber, 
the specialization map is the analogue of the map 
$$
\rho_t^*: K_\top^1(X^*) \to K_\top^1(L)
$$
occurring in the work of Buchweitz and van Straten. As mentioned, the map $\rho_t^*$ factors 
through the topological $K$-theory of the Milnor fiber. We prove the
analogous fact here, by showing, under certain hypotheses, that the specialization map 
$\sigma$ factors through the $G$-theory of the  algebraic analogue of the Milnor fiber.
Our notion of specialization in $K$- and $G$-theory is based on work of Andrei Suslin 
\cite{Suslin}.

Define a {\em pointed curve} to be a pair $(C,c)$, where $C$ is an affine, Noetherian, integral 
scheme that is regular of dimension one and
  $c \in C$ is a closed point such that the associated maximal ideal is principle and the 
residue field $\kappa(c)$ is algebraically closed. 
Equivalently, a pointed curve is a pair $(V, \fn)$ where $V$ is a Dedekind domain and $\fn$ is 
a principle maximal ideal of $V$ such that 
$V/\fn$ is an algebraically closed field. We will use both notations $(C,c)$ and $(V, \fn)$ 
interchangeably. We typically write $t$ for a
chosen generator of $\fn$, but such a choice is not part of the defining data.

A typical example occurs when $V$ is a dvr with algebraically closed residue field (and hence 
$t$ is a uniformizing parameter), but we do not limit ourselves to this case.

We will use the exact sequence
$$
K_1(C) \xra{i^*} K_1(C \sm c) \xra{\partial} K_0(c),
$$
equivalently, the exact sequence
$$
K_1(V) \xra{i^*} K_1(V\adi{t}) \xra{\partial} K_0(V/\fn),
$$
coming from a portion of the long exact localization sequence in $K$-theory. (Since $V$, 
$V\adi{t}$ and $V/\fn$  are all regular, their $K$- and $G$-theories coincide). 
The map $\partial$ sends $[t] \in K_1(V\adi{t})$ to $[V/\fn] \in K_0(V/\fn)$ and $[V/\fn]$ is a 
generator of $K_0(V/\fn) \cong \Z$.

Given a pointed curve $(C,c)$ and a morphism of Noetherian schemes $p: X \to C$, write 
$X_c$ for the fiber over $c \in C$ and $X \setminus X_c$ for its open
complement in $X$. Observe that if $V$ is a dvr, then $X_c$ is the closed fiber and $X 
\setminus X_c$ is the generic fiber. If $X = \Spec(Q)$ is affine, so
that $p$ is given by a ring map $V \to Q$, then $X_c = \Spec(Q/f)$ and $X \setminus X_c = 
\Spec(Q\adi{f})$, where $f \in Q$ is the image of a chosen generator
$t$ of $\fn$. 

Similarly, 
given a ring map $V \to Q$ and a maximal ideal $\fm$ of $Q$ such that $\fm \cap V = \fn$, 
then setting $X = \Spec(Q) \sm \fm$, we have $X_c = \Spec(Q/f) \sm
\fm$ and $X \setminus X_c = \Spec(Q\adi{f})$. This is the main example we have in 
mind, but for most of this section, 
we allow $X$ to be an arbitrary
Noetherian scheme.

\begin{defn} Given a pointed curve $(C,c)$ corresponding to $(V,\fn)$, a morphism of Noetherian schemes $p: X 
\to C$, and a
  positive integer $\ell$ such that $\chr(\kappa(c)) \nmid \ell$, 
define the {\em specialization map in $G$-theory} (for $X$ with $\Z/\ell$ coefficients) 
to be the map
$$
\sigma = \sigma_{C,c,p}:
G_i(X \setminus X_c, \Z/\ell) \to 
G_i(X_c, \Z/\ell)
$$
given as follows.
Choose $z$ to be any element of $K_1(V\adi{t})$ that maps to $1 \in \Z$ under $\partial: 
K_1(V\adi{t}) \to K_0(V/\fn) = \Z$; for example, $z$ could be $[t]$. 
Then $\sigma$ is the composition of 
$$
G_i(X \setminus X_c, \Z/\ell) \xra{- \cap z}
G_{i+1}(X \setminus X_c, \Z/\ell) \xra{\partial}
G_i(X_c, \Z/\ell)
$$
where $\partial$ is the boundary map in the localization long exact sequence in $G$-theory 
associated to the closed subscheme $X_c$ of $X$.
\end{defn}

\begin{lem}[Suslin] The specialization map $\sigma_{C,c,p}$ is independent of the choice 
of $z$.
\end{lem}

\begin{proof} Our proof is basically that of Suslin's, with some minor modifications.

Suppose $z'$ is another element of $K_1(V\adi{t})$ with $\partial(z') = 1$, and let 
$\sigma': G_i(X \setminus X_c, \Z/\ell) \to 
G_i(X_c, \Z/\ell)$ be the map $\sigma'(\alpha) = \partial(\a \cap z')$.
The difference $\sigma - \sigma'$ sends $\alpha
  \in G_i(U\setminus U_c, \Z/\ell)$ to $\partial(\alpha \cap (z - z'))$. Since $\partial(z-z') = 
0$, we have $z - z' = i^*(w)$ for some $w \in V$.
Using \cite[1.1]{Suslin}, we have 
$$
\partial(\alpha \cap i^*(w)) = \partial(\alpha) \cap j^*(w) \in G_i(U_c, \Z/\ell)
$$
where $j^*: K_1(V) \to K_1(V/\fn)$ is the map induced by modding out by $\fn$. Since we 
assume $V/\fn$ is algebraically closed and $\chr(V/\fn) \nmid \ell$, 
$j^*(w)$ is $\ell$-divisible, whence $\partial(\alpha) \cap j^*(w)  = 0$ since $G_i(U_c, 
\Z/\ell)$ is $\ell$-torsion. 
\end{proof}

\begin{rem} In fact, the previous result remains true (and the given proof remains valid) if the 
field $V/\fn$ is merely assumed to be closed under taking $\ell$-th roots.
\end{rem}

The specialization map is closely related to the map $\ttheta$. In detail, 
suppose $(V, \fn)$ is a pointed curve with  $\fn = (t)$ and $V \to Q$ is a map of rings 
sending $t$ to $f$. Let $\fm$ be a maximal ideal of $Q$ with $f \in \fm$
such that $\fm \cap V = \fn$. Set $X 
= \Spec(Q) \sm
  \fm$. Then $X \setminus X_c = \Spec(Q\adi{f})$, $X_c = \Spec(Q/f) \sm \fm$,  and  the 
specialization map $\sigma: G_i(Q\adi{f}, \Z/\ell) \to G_i(\Spec(Q/f) \sm \fm,
  \Z/\ell)$ (defined by setting $z = [t]$) is given by
$$
\sigma(\alpha) = \partial(\alpha \cap [f]).
$$
We have proven:

\begin{prop} \label{prop815}
Suppose $Q$ is a Noetherian ring, $\fm$ is a maximal ideal, $f \in \fm$ is non-zero-divisor,  and 
$\ell$ is a positive integer that is not divisible by $\chr(Q/\fm)$, $2$ or $3$.
Set $R = Q/f$. If there
exists a pointed curve $(V, \fn)$ and a ring map $V \to Q$ that sends some generator $t$ of 
$\fn$ to $f$, then for all integers $i,j$ the $\ttheta$ pairing 
fits into a commutative square
\begin{equation} \label{E815}
\xymatrix{
G_i(Q\adi{f}, \Z/\ell) \times K_j(\Spec(R) \sm \fm, \Z/\ell) \ar[d]^{\sigma \times \id} 
\ar[rr]^{\phantom{XX} \ttheta} && G_{i+j-1}(R/\fm, \ell) \\
G_i(\Spec(R) \sm \fm, \Z/ \ell) \times K_j(\Spec(R) \sm \fm, \Z/\ell) \ar[rr]_{\phantom{XXXXX} 
(- \cap -)} && G_{i+j}(\Spec(R) \sm \fm, \Z/\ell) \ar[u]^{\partial'} \\
}
\end{equation}
\end{prop}

\begin{ex}
If $Q$ contains a field $k$, then we may take $(V, \fn) = (k[t], (t))$ and define $V \to Q$ to be 
the unique $k$-algebra map sending $t$ to $f$.
\end{ex}

We will need to understand the behavior of the specialization map as we allow the base 
$(C,c)$ to vary in a suitably nice manner.

\begin{defn}
A {\em morphism of pointed curves}, say from $(C',c')$ to $(C,c)$, is a flat morphism of 
schemes $\phi: C' \to C$ such that $\phi^{-1}(c) = \{c'\}$ and the induced
map on residue fields is  an isomorphism: $\kappa(c) \xra{\cong} \kappa(c')$. Equivalently,  if 
$(C,c)$ corresponds to $(V,\fn)$ and $(C',c')$ to $(V', \fn')$, 
then a morphism is a flat ring map $g: V \to V'$
such that $\fn V'$ is $\fn'$-primary and the induced map $V/\fn \xra{\cong} V'/\fn'$ is an 
isomorphism. 
\end{defn}

Since $g$ is necessarily injective, we often just write a morphism as if it were a ring 
extension: $V \subseteq V'$.
If $\fn = (t)$ and $\fn = (t')$, then since $V$ and $V'$ are Dedekind domains, the condition 
that $\fn V'$ is $\fn'$-primary
is equivalent to the existence of an equation $t =u' (t')^n$ for some integer $n \geq 1$ and unit $u' \in
(V')^\times$.

\begin{defn} 
A morphism of pointed curves is {\em unramified} if $C' \times_C \Spec (\kappa(c)) \cong 
\kappa(c')$ (via the canonical map), or, equivalently,  if $\fn V =
\fn'$. If $\fn = (t)$ and $\fn' = (t')$, being unramified is equivalent to $t = u't'$ for some $u' 
\in (V')^\times$. 

A morphism of pointed curves is {\em finite} if the underlying map of schemes $\phi: C' \to 
C$ is finite, or, equivalently, $V'$ is a finitely generated
$V$-module. Since $V$ and $V'$ are Dedekind domains, $V \subseteq V'$ is finite if and only 
if the induced map on fields of fractions $E \subseteq E'$ is finite
and $V'$ is the integral closure of $V$ in $E'$.
\end{defn}

If $\phi: (C',c') \to (C,c)$ is a morphism of pointed curves and $p: X \to C$ is a morphism of 
Noetherian schemes, we will write $p': X' \to C'$ for the
pull-back of $p$ along $\phi$. Abusing notation, we write $\phi: X' \to X$  for the induced 
map, and also use $\phi$ for the induced maps on fibers.
Observe that our hypotheses ensure that the induced map on fibers over the marked points is 
an isomorphism: $\phi: X'_{c'} \xra{\cong} X_c$.
Using also that $\phi^{-1}(c) = \{c'\}$, we have that both squares in 
$$
\xymatrix{
X' \sm X'_{c'} \ar[r] \ar[d]^{\phi} & X' \ar[d]^{\phi} & X'_{c'} \ar[l] \ar[d]_{\phi}^\cong \\
X \sm X_{c} \ar[r] & X & X_c \ar[l]
}
$$
are Cartesian.

\begin{lem}  \label{lem84a}
If $\phi: (C',c') \to (C,c)$ is a morphism pointed curves that is either finite or unramified, $p: X \to C$ is any 
morphism of Noetherian schemes, and $\ell$ is any integer, the diagram
$$
\xymatrix{
G_i(X' \sm X'_{c'}, \Z/\ell) \ar[rr]^{\partial'} \ar[d]_{\phi_*} && G_{i-1}(X'_{c'}, \Z/\ell) 
\ar[d]_{\phi_*}^\cong\\
G_i(X \sm X_{c}, \Z/\ell) \ar[rr]^{\partial} && G_{i-1}(X_{c}, \Z/\ell)  \\
}
$$
commutes.
\end{lem}

\begin{proof} This follows from the fact that
$$
\xymatrix{
\cG(X'_{c'}) \ar[r] \ar[d]^{\phi_*} & \cG(X') \ar[r] \ar[d]^{\phi_*} & \cG(X' \sm X'_{c'}) 
\ar[d]^{\phi_*} \\
\cG(X_{c}) \ar[r]         &  \cG(X) \ar[r]         & \cG(X \sm X_{c}) 
}
$$
is a homotopy commutative of spectra. 
\end{proof}

\begin{lem} \label{lem87}
If $\phi: (C',c') \to (C,c)$ is either a finite or an unramified morphism of pointed curves, $p: X 
\to C$ is any morphism of Noetherian schemes and $\ell$ is
a positive integer not divisible by $\chr(\kappa(c))$, the diagram
$$
\xymatrix{
G_i(X \sm X_{c}, \Z/\ell) \ar[rr]^{\sigma} \ar[d]_{\phi^*} && G_{i}(X_{c}, \Z/\ell) 
\ar[d]_{\phi^*}^\cong\\
G_i(X' \sm X'_{c'}, \Z/\ell) \ar[rr]^{\sigma'} && G_{i}(X'_{c'}, \Z/\ell)  \\
}
$$
commutes, where $\sigma = \sigma_{C,c,p}$ and $\sigma' = \sigma_{C',c',p'}$ are the 
specialization maps.
\end{lem}

\begin{proof}
Assume $\phi$ is finite. 
Since $\phi_*: G_i(X_c, \Z/\ell) \xra{\cong} G_i(X'_{c'}, \Z/\ell)$ is the inverse of $\phi^*$, it 
suffices to prove
$$
\xymatrix{
G_i(X \sm X_{c}, \Z/\ell) \ar[rr]^{\sigma} \ar[d]_{\phi^*} && G_{i}(X_{c}, \Z/\ell) \\
G_i(X' \sm X'_{c'}, \Z/\ell) \ar[rr]^{\sigma'} && G_{i-1}(X'_{c'}, \Z/\ell) \ar[u]_{\phi_*}^\cong \\
}
$$
commutes.
Choose any $z' \in K_1(V'\adi{t'})$ as in the definition of the specialization map $\sigma'$. A 
special case of Lemma \ref{lem84a} gives that $\partial
\phi_*(z') = 1$,
and hence we may choose $z := \phi_*(z')$ in the definition of $\sigma$. Using Lemma 
\ref{lem84a} again gives
$$
\phi_* \sigma'(\phi^* \alpha) =
\phi_* \partial' (\phi^*(\alpha) \cap z') =
\partial \phi_*(\phi^*(\a) \cap  z').
$$
By the Projection Formula, $\phi_*(\phi^*(\a) \cap  z') = \a \cap \phi_*(z') = \a \cap z$, since 
$z = \phi_*(z')$. The result follows.

Now assume  $\phi$ is unramified. Then
$$
\xymatrix{
X'_{c'} \ar[r] \ar[d]^\cong & X' \ar[d] \\
X_c \ar[r] & X 
}
$$
is a Cartesian square, and it follows that
$$
\xymatrix{
\cG(X_c) \ar[r] \ar[d]^{\phi^*} & \cG(X) \ar[r] \ar[d]^{\phi^*} & \cG(X \setminus X_c) 
\ar[d]^{\phi^*} \\
\cG(X'_{c'}) \ar[r] & \cG(X') \ar[r] & \cG(X' \setminus X'_{c'}) 
}
$$
is homotopy commutative diagram. The boundary maps $\partial, \partial'$ in the associated 
long exact sequences of homotopy
groups (with any coefficients) thus commute with $\phi^*$.

Let $t \in \fn$ be any generator. Since $\phi$ is unramified, $t' := \phi(t)$ generates $\fn'$.
Using $z = [t]$ and $z' = [t']$ for the definition of $\sigma$ and $\sigma'$, 
we have for any $\a \in G_i(X \setminus X_c, \Z/\ell)$
$$
\sigma(\a) = \phi^* \partial(\a \cap [t]) = \partial'(\phi^*(\a \cap [t])) = \partial' (\phi^*(a) 
\cap [t']) = \sigma'(\phi^*(\a)),
$$
since $\phi^*$ commutes with $\cap$ product and $\phi^*([t]) = [t']$.
\end{proof}

The notion of the Henselization of a ring at a maximal ideal will
  be important for the rest of this paper. Let us recall  its main properties. Given a Noetherian ring $Q$ and a maximal ideal $\fm$ of it, 
we write $Q^\hen_\fm$ for the Henselization of $Q$ at $\fm$.
We have:
\begin{itemize}
\item $Q^\hen_\fm$ is a Noetherian local ring, with maximal ideal $\fm^\hen$, and the pair 
$(Q^\hen, \fm^\hen)$ satisfies Hensel's Lemma: Given a monic
  polynomial $p(x) \in Q^\hen_\fm[x]$, if its image $\o{p}(x) \in Q^\hen_\fm/\fm^\hen[t]$ 
has a simple root $a$, then $p(x)$ has a root $\a \in Q^\hen_\fm$ whose
  image in $Q^\hen_\fm/\fm^\hen$ is $a$.

\item There are flat ring maps $Q \to Q_\fm \to Q^\hen_\fm$ and $Q_\fm \to Q^\hen_\fm$ 
is faithfully flat. 

\item We have $\fm Q^\hen_\fm = \fm^\hen$, or, in other words, the fiber of 
$\Spec(Q^\hen_\fm) \to \Spec(Q)$ over $\fm$ is $\Spec(\kappa(\fm^\hen)) = 
\Spec(Q^\hen_\fm/ \fm^\hen)$.

\item $Q_\fm$ is regular if and only if $Q^\hen_\fm$ is regular.

\item \cite[18.6.9]{EGA4d} The fibers of $\Spec(Q^\hen_\fm) \to \Spec(Q)$ are spectra of 
finite products of algebraic, separable field extensions. That is, for every $\fp
  \in \Spec(Q)$, there are a finite number of primes $\fq_1, \dots, \fq_d$ in $Q^\hen_\fm$ 
such that $\fq_i \cap Q = \fp$ and, for each $i$, the induced field map 
$\kappa(\fp) \into \kappa(\fq_i)$ is separable algebraic.
\end{itemize}

For any pointed curve $(V, \fn)$, let $V^\hen$ denote the Henselization of $V$ at 
$\fn$ and let $\fn^\hen$ denote also the maximal
ideal of $V^\hen$.  Then $V^\hen$ is also a Dedekind domain and $(V, \fn) \to (V^\hen, 
\fn^\hen)$ is a unramified morphism of
pointed curves.

In the following Theorem, given a morphism $p: X \to \Spec(V)$ of Noetherian schemes, we will 
consider $G$-theory of the geometric generic fiber
$$
X_{\o{F}} := X \times_{\Spec(V)} \Spec(\o{F}),
$$
where $F$ is the field of fractions of a Henselian ring $V$ 
and $\o{F}$ is its algebraic closure.
Note that $X_{\o{F}}$ may fail to be Noetherian and $G$-theory is ordinarily defined for just 
Noetherian schemes. (For non-Noetherian schemes, the
category of coherent sheaves may fail to be an abelian category.) 
However, since $X_L := X \times_{\Spec{V}} \Spec(L)$ is Noetherian for every finite field 
extension $L$ of $F$,
we have that $X_{\o{F}}$ is a filtered colimit of Noetherian schemes with flat transition maps. 
We thus take, as definition, the $G$-theory spectrum of
$X_{\o{F}}$ to be the filtered colimit of spectra
$$
\cG(X_{\o{F}}) := \colim_L \cG(X_L)
$$
indexed by the finite field extensions of $F$ contained in $\o{F}$. The associated $G$-groups 
are thus also given by colimits
$$
G_n(X_{\o{F}}) = \colim_L G_n(X_L).
$$

\begin{rem}
At least when $X = \Spec(A)$ is affine, 
$G_n(X_L)$ may be interpreted as the $K$-theory of an abelian category.
Indeed, more generally, if $A$ is a filtered colimit, $A = \colim_{i \in I} A_i$, of Noetherian 
rings with flat transition maps, then $G_n(A) := \varinjlim_{i \in } G_n(A_i)$ is the
$K$-theory of the category of finitely presented $A$-modules and the latter forms an abelian 
category under these assumptions. 
\end{rem}

In later sections we will be mostly interested in the case where $X_F \to \Spec(F)$ is smooth, 
so that each transition map in the colimit giving $X_{\o{F}}$ is
a flat morphism of regular rings, and we will also be interested only in the case when $X$ is 
affine. In this situation, $G$-theory and 
$K$-theory coincide, as we now explain. 

More generally, suppose $A = \colim_{i \in I} A_i$ is a filtered colimit of regular Noetherian 
rings with flat transition maps.
Even if $A$ is not Noetherian, the correct notion of the $K$-theory of $A$ is unambiguous: it is the $K$-theory of 
the exact category $\cP(A)$ of finitely generated projective $A$-modules. 
Moreover, we have
$$
K_n(A) = \varinjlim_{i \in I} K_n(A_i)
$$
for all $n$. Since we assume each $A_i$ is regular Noetherian, each natural map $K_n(A_i) \to 
G_n(A_i)$ is an isomorphism and thus the canonical map
$$
K_n(A) \xra{\cong} G_n(A)
$$
is an isomorphism. In particular, $K_n(X_{\o{F}}) \cong G_n(X_{\o{F}})$ provided $X_F \to 
\Spec(F)$ is smooth and  $X$ is affine.

\begin{thm}  \label{thm815}
Given a pointed curve $(V, \fn)$ and a  morphism $p: X \to \Spec(V)$
  of Noetherian schemes, let $X_{\o{F}} = X \times_{\Spec(V)} \Spec(\o{F})$, where $\ov{F}$ is 
the algebraic closure of the
field of fractions of the Henselization $V^\hen$ of $V$ at $\fn$.
For any positive integer $\ell$ such that $\chr(\kappa(c)) \nmid \ell$, 
the specialization map $\sigma: G_i(X \sm X_c, \Z/\ell) \to G_i(X_c, \Z/\ell)$ factors through 
$G_i(X_{\o{F}})$; i.e., there is a commutative triangle
$$
\xymatrix{
G_i(X \sm X_c, \Z/\ell) \ar[rd]  \ar[rr]^{\sigma_{C,c,p}} && G_i(X_c, \Z/\ell). \\
& G_i(X_{\o{F}}, \Z/\ell) \ar[ru]\\
}
$$

Moreover, if $\ell$ is also not divisible by $2$ or $3$, then the direct sum indexed  over  $i 
\geq 0$ of these morphisms are graded $K_*(V, \Z/\ell)$-module homomorphisms.
\end{thm}

\begin{proof} As discussed above, the map $(V,\fn) \to (V^\hen,\fn)$ is unramified and so 
using Lemma \ref{lem87}, we may assume without loss of generality that
  $(V,\fn)$ is Henselian with field of fractions $F$. Note that the integral closure 
$\overline{V}$ of $F$ is a filtered colimit  of rings of the form $V'$
  where $V'$ is the integral closure of $V$ in some finite field extension $F \subseteq F'$ of the 
field of fractions of $V$. Using Lemma \ref{lem87} again, we
  have a commutative triangle
$$
\xymatrix{
G_i(X \sm X_{c}, \Z/\ell) \ar[rr]^{\sigma} \ar[d] && G_{i}(X_{c}, \Z/\ell) \\
G_i(X' \sm X'_{c'}, \Z/\ell) \ar[rru]_{(\phi^*)^{-1} \circ \sigma'}
}
$$
for each such $V'$. Moreover, given maps $V \into V' \into
V''$ associated to a chain of finite field extensions $F \subseteq F' \subseteq F''$, the diagram
$$
\xymatrix{
G_i(X \sm X_{c}, \Z/\ell) \ar[rr]^{\sigma} \ar[d] && G_{i}(X_{c}, \Z/\ell) \\
G_i(X' \sm X'_{c'}, \Z/\ell) \ar[rru]^{\sigma'} \ar[d]\\
G_i(X'' \sm X''_{c''}, \Z/\ell) \ar[rruu]^{\sigma''}
}
$$
commutes. Taking colimits thus gives a commutative triangle
$$
\xymatrix{
G_i(X \sm X_{c}, \Z/\ell) \ar[r]^{\sigma} \ar[d] & G_{i}(X_{c}, \Z/\ell). \\
\colim_{F \subseteq F'} G_i(X' \sm X'_{c'}, \Z/\ell) \ar[ru]
}
$$
The first result follows, since $\colim_{F \subseteq F'} (X' \sm X'_{c'}) \cong X_{\ov{F}}$  
and 
$$
G_i(X_{\o{F}}, \Z/\ell) := \colim_{F \subseteq F'} G_i(X' \sm X'_{c'}, \Z/\ell).
$$

For the final claim, note that every morphism occurring in this proof commutes with 
multiplication by any fixed element of $K_*(V, \Z/\ell)$. 
\end{proof}

\begin{cor} \label{cor87}
Suppose $Q$ is a Noetherian ring, $\fm$ is a maximal ideal, $f \in \fm$ is a non-zero-divisor, 
and $\ell$ is a 
positive integer that is not divisible by $\chr(Q/\fm)$,
$2$ or $3$. Set $R = Q/f$. If there
exists a pointed curve $(V, \fn)$ and a ring map $V \to Q$ that sends some generator $t$ of 
$\fm$ to $f$, then for all $i,j$, the $\ttheta$ pairing fits into a commutative
square:
$$
\xymatrix{
G_i(Q\adi{f}, \Z/\ell) \times K_j(\Spec(R) \sm \fm, \Z/\ell) \ar[r]^{\phantom{XXX} \ttheta}  
\ar[d] & K_{i+1-1}(R/\fm, \Z/\ell) \\
 G_i(Q^\hen_\fm \otimes_V \overline{F}, \Z/\ell) \times K_j(\Spec(R) \sm \fm, \Z/\ell) \ar[ru] 
\\
}
$$
where $Q^\hen_\fm$ is the Henselization of $Q$ at $\fm$ and 
$\o{F}$ is the algebraic closure of the field of fractions $F$ of $V^\hen$.
Moreover, the direct sum indexed by $i,j \geq 0$ of these morphisms are graded $K_*(V, 
\Z/\ell)$-module homomorphisms.

If $F \to Q \otimes_V F$ is smooth, then 
$$
G_i(Q^\hen_\fm \otimes_V \overline{F}, \Z/\ell) \cong
K_i(Q^\hen_\fm \otimes_V \overline{F}, \Z/\ell).
$$
\end{cor}

\begin{proof}
Applying Proposition \ref{prop815b} to the flat map $Q \to Q^\hen_\fm$ allows us to reduce 
to the case where $Q = Q^\hen_\fm$. The result then 
follows from Proposition \ref{prop815} and Theorem \ref{thm815}. 
\end{proof}

\begin{rem} In fact, Corollary \ref{cor87} remains valid even if $\ell = 0$ (i.e., with $\Z$-
coefficients). Since we shall not need this fact in this paper, we omit its
  proof.
\end{rem}

\section{\'Etale (Bott inverted) $K$-theory}
Recall that the vanishing result of Buchweitz and van Straten uses topological $K$-theory, not algebraic $K$-theory. In a general characteristic setting, the
best replacement for topological $K$-theory is \etale $K$-theory (with finite coefficients). 
This leads us to the goal of proving that, under suitable hypotheses,  the pairing
$$
\ttheta: G_0(Q\ad{f}, \Z/\ell) \times K_1(\Spec(R)
\setminus \fm , \Z/\ell) \to \Z/\ell
$$
factors through the analogous pairing involving 
\'etale $G$- and $K$-theory. 

This is roughly what we achieve in this section. 
But, to avoid some nagging issues in 
the foundations of \'etale $K$-theory, it proves simpler to use instead the theory obtained from
algebraic $K$-theory (and $G$-theory) with finite coefficients by inverting the so-called ``Bott element''.
Since the resulting theory is closely related to topological $K$-theory, we will write it as $\Ktop$ (and $\Gtop$); see Definition \ref{Def1127} below.
Motivation for our approach is provided by Thomason's Theorem
\cite[4.1]{Thomason} (recently extended by Rosenschon-{\O}stv{\ae}r \cite[4.3]{OR}), 
which says, roughly, that ``Bott inverted algebraic $K$-theory with
$\Z/\ell$ coefficients and
\'etale $K$-theory 
with $\Z/\ell$ coefficients coincide''.

\begin{defn} \label{Def1127}
For any scheme $X$ and prime $\ell$,  let $\cK(X, \Z/\ell)$ be the result of 
smashing the algebraic $K$-theory spectrum of $X$, $\cK(X)$,  with the mod-$\ell$
Moore space, and for any Noetherian scheme $X$, let $\cG(X, \Z/\ell)$ be the result of smashing 
$\cG(X)$ with the mod-$\ell$ Moore space.

Let $\cKU$ denote the spectrum representing (two periodic) topological 
$K$-theory and let $L_\cKU \cE$ denote the Bousfield localization
of a spectrum $\cE$ at $\cKU$.

Finally, for any scheme $X$ and prime $\ell$, define
$$
\Ktop_n(X, \Z/\ell) := \pi_n L_\cKU \cK(X, \Z/\ell), n \in \Z,
$$
and for any Noetherian scheme $X$, define
$$
\Gtop_n(X, \Z/\ell) := \pi_n L_\cKU \cG(X, \Z/\ell), n \in \Z.
$$
\end{defn}

Let us bring these definitions down to Earth a bit. Assume for simplicity that $\ell \geq 5$ and 
define $\mu_\ell = e^{2 \pi i/\ell}$,  a primitive $\ell$-th root of unity in $\C^\times$, and
consider the ring $\Z[\mu_\ell]$.
Recall (e.g., from \cite[Appendix A]{Thomason}) that the {\em Bott element} is
a certain canonical element $\beta \in K_2(\Z[\mu_\ell], \Z/\ell)$, which maps to 
$[\mu_\ell] \in K_1(\Z[\mu_\ell])$ under the boundary map. 
If $X$ is a scheme over 
$\Spec(\Z[\mu_\ell])$, then we obtain by pullback a Bott element
$\beta \in K_2(X, \Z/\ell)$.

For example, if $X = \Spec(k)$ is an algebraically closed field with $\chr(k) \ne \ell$, then 
there is a map $\Z[\mu_\ell] \to k$, specified by a choice of a primitive $\ell$-th root of
unity for $k$. Moreover,
the boundary map
of the long exact coefficients sequence for the $K$-theory of $k$ determines (using \cite{SuslinLocalFields, Suslin})
an  isomorphism
$$
K_2(k, \Z/\ell) \xra{\cong} \mu_\ell(k),
$$
and a Bott element $\beta$ for $k$ maps to the chosen primitive $\ell$-th root of unity in 
$k$ under this map.
More generally, if $V$ is a Henselian dvr with algebraically closed residue field $k$ such that 
$\chr(k) \ne \ell$, then we have 
$$
K_i(V, \Z/\ell) \cong K_i(k, \Z/\ell) 
\cong
\begin{cases}
\Z/\ell, & \text{if $i$ is even and} \\
0, & \text{otherwise,}
\end{cases}
$$
by \cite{SuslinLocalFields, Suslin}, 
and a Bott element for $V$ may also be specified by choosing a primitive $\ell$-th root of 
unity in $k$.

Given a Bott element $\beta \in K_2(\Z[\mu_\ell], \Z/\ell)$ for some prime $\ell \geq 5$, the 
element $\beta$ acts on the $K$- and $G$-groups of all schemes
over $\Z[\mu_\ell]$, and this action is compatible with the localization long exact sequences 
for such schemes and it commutes with cup and cap products. Thus,
formally inverting the action of $\beta$ preserves all the structure needed in this paper.

\begin{lem} Let $\ell \geq 5$ be a prime, let $X$ be a scheme over $\Z[\mu_\ell]$, and let 
$\beta \in K_2(X, \Z/\ell)$ be the associated Bott element for $X$. 
There is a  natural isomorphism
$$
\Ktop_*(X, \Z/\ell) \cong K_*(X, \Z/\ell)\ai{\beta},
$$
where $K_*(X, \Z/\ell)\ai{\beta}$ denotes the graded ring obtained by inverting the 
homogeneous central element $\beta$ in the graded ring $K_*(X, \Z/\ell)$.

If $X$ is Noetherian, there is a natural isomorphism
$$
\Gtop_*(X, \Z/\ell) \cong G_*(X, \Z/\ell)\ai{\beta},
$$
where $G_*(X, \Z/\ell)\ai{\beta}$ is the localization of the graded $K_*(X, \Z/\ell)$-module
$G_*(X, \Z/\ell)$ by $\beta$.
\end{lem}

\begin{proof} See \cite[A.14]{Thomason}.
\end{proof}

\begin{rem} Since $\beta$ has degree two, one can view
  $\Ktop_*$ and $\Gtop_*$ as $\Z/2$-graded abelian groups.
\end{rem}

If $V$ is a Henselian dvr with algebraically closed residue field $k$ 
such that $\chr(k) \ne \ell$, 
$Q$ is a $V$-algebra, $\fm$ is a maximal ideal of $Q$, and $f \in \fm$ is a non-zero-divisor,
then the family of pairings 
$$
\ttheta: G_i(Q\ad{f}, \Z/\ell) \times K_j(\Spec(Q/f) \setminus \fm,
\Z/\ell) \to
K_{i+j-1}(Q/\fm, \Z/\ell). 
$$
forms a pairing of graded 
$K_*(V, \Z/\ell)$-modules, and hence upon inverting the action of the Bott
element $\beta \in K_2(V, \Z/\ell)$,  we
obtain the pairing
$$
\ttheta^\top:
\Gtop_i(Q\ad{f}, \Z/\ell) \times \Ktop_j(\Spec(Q/f) \setminus \fm,
\Z/\ell) \to
\Ktop_{i+j-1}(Q/\fm, \Z/\ell). 
$$
If $Q/\fm$ is algebraically closed (e.g., if $Q$ is a
finite type $V$-algebra or the localization of such at a maximal ideal), then by \cite{SuslinLocalFields, Suslin} we have
$$
K_*(Q/\fm, \Z/\ell) = \Z/\ell[\beta]
$$
and hence
$$
\Ktop_*(Q/\fm, \Z/\ell)  = \Z/\ell[\beta,
\beta^{-1}].
$$
In particular, $K_0(Q/\fm, \Z/\ell) \cong \Ktop_0(Q/\fm, \Z/\ell) \cong \Z/\ell$.

\begin{prop} \label{Prop1025}
Assume $V$ is a dvr with algebraically closed residue field $k$ and field of fractions $F$, $Q$ 
is a
flat $V$-algebra, $t$ is a uniformizing parameter of $V$ that maps to an
element $f \in Q$, and $\fm$ is a maximal ideal of $Q$ containing $f$ such that $Q/\fm$ is 
algebraically closed. Set $R = Q/f$.
For any prime $\ell \geq 5$ not divisible by $\chr(k)$, 
there is a commutative diagram
$$
\xymatrix{
G_1(Q\ad{f}, \Z/\ell) \times K_0(\Spec(R) \smm, \Z/\ell) \ar[drr]^{\ttheta}
\ar[d] && \\
\Gtop_1(Q\ad{f}, \Z/\ell) \times \Ktop_0(\Spec(R) \smm, \Z/\ell)
\ar[rr]^{\phantom{XXXXXXX} \ttheta^\top} \ar[d] && \Z/\ell. \\
\Gtop_1(Q^\hen_\fm \otimes_V \overline{F}, \Z/\ell) \times \Ktop_0(\Spec(R) \smm, \Z/\ell)
\ar[urr] && \\
}
$$
\end{prop}

\begin{proof} 
This diagram is obtained from the commutative diagrams in Corollary
\ref{Cor3} and Corollary \ref{cor87} by inverting the action of $\beta \in K_2(V, \Z/\ell)$.
\end{proof}

\begin{cor} \label{Cor98a} 
With the assumptions of Proposition \ref{Prop1025}, if 
$\Gtop_1(Q^\hen_\fm \otimes_V \overline{F}, \Z/\ell) = 0$
for infinitely many 
primes $\ell \geq 5$, then $\theta^R(M,N) = 0$ for all
finitely generated $R$-modules $M$ and $N$ such that $N_\fp$ has
finite projective dimension for all $\fp \ne \fm$.
\end{cor}

\begin{proof} The Proposition implies that  $\theta^R(M,N)$ is a multiple of $\ell$ for 
infinitely many primes $\ell$.
\end{proof}

\section{Vanishing of $\Ktop_1$}

In this section, we combine a theorem of Rosenschon-{\O}stv{\ae}r \cite[4.3]{OR}, which is an 
improvement of a theorem of Thomason \cite[4.1]{Thomason}, and a
theorem of 
Illusie \cite[2.10]{Illusie}  to
establish the vanishing of the odd-degree topological $K$-groups
with $\Z/\ell$-coefficients of the 
``algebraic Milnor fiber'', $\Spec(Q^\hen_\fm \otimes_V \overline{F})$, in certain cases.

\subsection{The Thomason-Rosenschon-{\O}stv{\ae}r  Theorem}

\begin{defn} Fix a prime $\ell$. 
The {\em mod-$\ell$ \etale cohomological dimension} of a scheme $X$, written $\cd_\ell(X) \in \N \cup \{\infty\}$,  is defined
as 
$$
\cd_\ell(X) = \sup\{i \, | \, H^i_{\et}(X, \cF) \ne 0 \text{ for some $\ell$-power torsion \etale 
sheaf $\cF$}\}.
$$
\end{defn}

If $F$ is algebraically closed (or even just separably closed), 
then $\cd_\ell(F) = 0$ 
since the \etale topology of such a field is
trivial. If $E$ is a field extension of $F$ of transcendence degree $d$, then by \cite[II.4.2.11]
{Serre} we have
$$
\cd_\ell(E) \leq \cd_\ell(F) + d.
$$
It follows that if $E$ is a field of transcendence degree $M$ over an algebraically closed field, then
\begin{equation} \label{E817}
\cd_\ell(E) \leq M.
\end{equation}

\begin{defn} For an odd prime $\ell$,
a scheme $X$ is {\em $\ell$-good} if $X$ is quasi-separated, quasi-compact and  of finite 
Krull dimension, $\ell$ is a unit of $\Gamma(X,
  \cO_X)$, and there is a uniform finite bound on the mod-$\ell$ \etale cohomological 
dimensions of all the residue fields of $X$. A commutative ring
  $A$ is {\em $\ell$-good} if $\Spec(A)$ is $\ell$-good.
\end{defn}

\begin{rem} The definition of $\ell$-good is motivated by the hypotheses of the 
  Rosenschon-{\O}stv{\ae}r Theorem. 
The correct version of $\ell$-good for $\ell = 2$ involves ``virtual \etale 
cohomological dimension''. 
\end{rem}

Note that an affine scheme $X = \Spec(A)$ is automatically separated (and hence quasi-separated) 
and quasi-compact. So, $A$ is $\ell$-good if
and only if $A$ has finite Krull dimension, $\ell$ is a unit of $A$, and there is a uniform finite 
bound on $\cd_\ell(\kappa(\fp))$ as $\fp$ ranges over all primes
ideals of $A$.

If $X$ is a separated scheme of finite type over an algebraically closed field $k$, then the 
transcendence degree of each of its residue fields is
at most $\dm(X) < \infty$, and hence we have $\cd_\ell(\kappa(x)) \leq \dm(X)$ for all $x \in 
X$. It follows that, for such $X$,  if $\ell \ne \chr(k)$, then $X$
is $\ell$-good.

We will need a generalization of this fact that involves Henselizations. Recall that given a 
commutative ring $V$, a $V$-algebra $Q$ is {\em essentially of
  finite type} over $V$ is $Q$ is the localization of a finitely generated $V$-algebra by some 
multiplicatively closed subset. 

\begin{lem} \label{lem817b}
Suppose $V$ is a Noetherian ring, $\fp$ is any prime of $V$, 
$Q$ is a $V$-algebra essentially of finite type, and $\fm$ is a maximal ideal of $Q$. Let 
$Q^\hen_\fm$ denote the
  Henselization of $Q$ at $\fm$. If $\ell \ne \chr(\kappa(\fp))$, then
$Q^\hen_\fm \otimes_V \o{\kappa(\fp)}$ is $\ell$-good, where $\o{\kappa(\fp)}$ is the 
algebraic closure of $\kappa(\fp)$. 
\end{lem}

\begin{rem}
The lemma generalizes to schemes: If $S$ is a Noetherian scheme, $p: X \to S$ is a morphism 
of schemes essentially of finite type, and
$x \in X$ is any closed point, then each geometric fiber of $X^\hen_x \to S$ is $\ell$-good 
for all $\ell$ not equal to the characteristic of the fiber.
\end{rem}

\begin{proof}
Since $Q$ is essentially of finite type over $V$ and $V$ is Noetherian, $Q$ is also Noetherian. 
It follows \cite[18.6.6]{EGA4d} that $Q^\hen_\fm$ is Noetherian, and since
it is local, $\dm(Q^\hen_\fm) < \infty$.

Again using that $Q$ is
essentially of finite
type over $V$, there is a bound $M$ such that for each prime $\fq \in \Spec(Q)$, the 
transcendence degree of $\kappa(\fq)$ over $\kappa(\fq \cap V)$ is
at most $M$. 
The residue fields of $Q^\hen_\fm$ are algebraic (and separable) over the corresponding 
residue fields of $Q$ by \cite[18.6.9]{EGA4d}. 
It follows that for each prime $\fq \in \Spec(Q^\hen_\fm \otimes_V \kappa(\fp))$, the 
transcendence degree of $\kappa(\fq)$ over $\kappa(\fp)$ is at most $M$.

To simplify notation, we state and prove a more general assertion: If $k$ is any field,  $A$ is a 
commutative $k$-algebra of finite Krull dimension, and there
exists a finite bound $M$ such that the transcendence degree of $\kappa(\fq)$ over $k$ is at 
most $M$ for all $\fq \in \Spec(A)$, then $A \otimes_k \o{k}$ is
$\ell$-good for all $\ell \ne \chr(k)$. The Lemma follows from the case $k = \kappa(\fp)$ 
and $A = Q^\hen_\fm \otimes_V \kappa(\fp)$.

The field $\o{k}$ is the filtered colimit of the finite extensions $L$ of $k$ contained in it, and 
hence $A \otimes_k \o{k}$ is the filtered colimit of the
collection of rings $\{A
\otimes_k L\}$. For each $L$, each residue field of $A \otimes_k L$ is a finite extension of the 
corresponding residue field of $A$.
As with any colimit of rings, the residue field of $A \otimes_k \o{k}$ at a prime $\fq$ is the
filtered colimit of the residue fields $\fq \cap (A \otimes_k L)$. We conclude that for each $\fq 
\in \Spec(A \otimes_k \o{k})$ the residue field
$\kappa(\fq)$  is 
algebraic over $\kappa(\fq \cap A)$ and hence has transcendence degree at most $M$ over 
$k$. Since $\kappa(\fq)$ contains $\o{k}$, it has
transcendence degree at most $M$ over $\o{k}$ as well. By \eqref{E817} we have
$\cd_\ell(\kappa(\fq)) \leq M$
for all $\fq \in \Spec(A \otimes_k \o{k})$. 

Finally, since $A \subseteq A \otimes_k \o{k}$ is an integral extension, $\dm(A \otimes_k 
\o{k}) = \dm(A) < \infty$. 
\end{proof}

The following result of Rosenschon-{\O}stv{\ae}r is an improvement of a celebrated theorem 
of Thomason \cite[4.1]{Thomason}. 

\begin{thm}[Rosenschon-{\O}stv{\ae}r] \label{thm:OR}    \cite[4.3]{OR} 
If $X$ is an $\ell$-good scheme for a prime $\ell \geq 5$, there is a strongly convergent, 
right half-plane spectral sequence
$$
E_2^{p,q} \Longrightarrow \Ktop_{q-p}(X, \Z/\ell)
$$
where
$$
E_2^{p,q} = 
\begin{cases} 
H^p_\et(X, \mu_{\ell}^{\otimes i}) & \text{ if $q = 2i$ and} \\
0 & \text{ if $q$ is odd,}
\end{cases}
$$
and the differential on the $r$-th page has bi-degree $(r,r-1)$: $d_r: E_r^{p,q} \to E_r^{p+r, 
q+r-1}$.
\end{thm}

\begin{rem} Some remarks are in order:
\begin{enumerate}
\item In their original paper, the abutment of this spectral sequence is $\pi_{q-p} L_\cKU 
\cK^B(X, \Z/\ell)$ where $\cK^B$ denotes {\em Bass's algebraic
    $K$-theory spectrum}. Since $\cK \to \cK^B$ induces an isomorphism on non-negative 
homotopy groups, the natural map $L_\cKU \cK \xra{\sim} L_\cKU \cK^B$ is
  a weak equivalence.

\item
The integer $i$ in $\mu_\ell^{\otimes i}$ is allowed to be negative. For $i < 0$, 
$\mu^{\otimes i} = (\mu_\ell^{-1})^{\otimes |i|}$ where
  $\mu^{-1}_\ell$ is the $\Z/\ell$-linear dual of $\mu_\ell$.

\item If $X$ is a scheme over an algebraically closed field of characteristic not equal to $\ell$, 
then upon choosing a primitive $\ell$-th root of unity, we may
identify $\mu_\ell^{\otimes i}$ with $\Z/\ell$, for all $i \in \Z$.

\item We have assumed $\ell \geq 5$ only to avoid some technical complications and 
because the cases $\ell = 2, 3$ will not be important for our
  purposes. But, appropriately interpreted, this Theorem remains valid for $\ell \in \{2,3\}$.
\end{enumerate}
\end{rem}

The following result gives the special case of the Theorem that we will need:

\begin{cor} \label{cor817c}
Assume $V$ is a Noetherian domain with field of fractions $F$, 
$Q$ is essentially of finite type over  $V$, the generic fiber of $\Spec(Q) \to \Spec(V)$ 
(namely, $\Spec(Q \otimes_V F) \to \Spec(F)$) is essentially smooth,
$\fm$ is a maximal ideal of $Q$, and
$\ell$ is a prime such that $\ell \geq 5$ and $\ell \ne \chr(F)$.
Then there is a strongly convergent spectral sequence
$$
E_2^{p,q} \Longrightarrow \Ktop_{q-p}(Q^\hen_\fm \otimes_V \o{F}, \Z/\ell)
$$
where 
$$
E_2^{p,q} = 
\begin{cases} 
H^p_\et(Q^\hen_\fm \otimes_V \o{F}, \mu_{\ell}^{\otimes i}) & \text{ if $q = 2i$ and} \\
0 & \text{ if $q$ is odd,}
\end{cases}
$$
and the differential on the $r$-th page has bi-degree $(r,r-1)$.
\end{cor}

\begin{proof} This follows from Lemma \ref{lem817b} and Theorem \ref{thm:OR}.
\end{proof}

\subsection{Illusie's Theorem}

We will need to make the  following assumptions:

\begin{assumptions} \label{ass} 
Assume $(V,k,F,Q,\fm,f,\ell)$ satisfy:

\begin{enumerate}

\item $V$ is a Henselian dvr with algebraically closed residue field
  $k$ and field of fractions $F$. \label{Qass1}

\item $Q$ is a regular ring, $\fm$ is a maximal ideal of $Q$, and $f \in \fm$. \label{Qass2}

\item There is a flat, finite type map $\Spec(Q) \to \Spec(V)$ of affine schemes of relative 
dimension $n$ such that the associated map of rings sends some uniformizing parameter
  $t \in V$ to $f \in Q$.  \label{Qass3}

\item The morphism $\Spec(Q) \to \Spec(V)$ is smooth at every point except, possibly, $\fm 
\in \Spec(Q)$. Notice in particular that the generic fiber $\Spec(Q \otimes_V F) \to
  \Spec(F)$ is smooth.
\label{Qass4} 

\item The morphism $\Spec(Q) \to \Spec(V)$ is a complete intersection near $\fm$ --- that
  is, for some $g \in Q \setminus\fm$, $Q \ai{g}$ the quotient of a smooth
  $V$-algebra by a regular sequence. \label{Qass5}

\item $\ell$ is prime not equal to $2$, $3$ or $\chr(k)$. Notice this implies $\ell \ne \chr(F)$ 
too.\label{Qass6}

\end{enumerate}
\end{assumptions}

\begin{thm}[Illusie's Theorem] 
\cite[2.10]{Illusie} Under Assumptions \ref{ass}, 
$$
H_\et^j( (Q^{\hen}_\fm) \otimes_V {\overline{F}}, \Z/\ell) = 0
$$
if $j \notin \{0,n\}$, where $Q^\hen_\fm$ is the Henselization of $Q$ at $\fm$ and $\o{F}$ is 
the algebraic closure of $F$.
\end{thm}

\begin{rem}
Note that Illusie's Theorem is the analogue of Milnor's Theorem, stating that the Milnor fiber 
of an analytic isolated singularity is homotopy equivalent to a  bouquet of
$n$-dimensional spheres. 
\end{rem}

\begin{cor} \label{cor89}
Under Assumptions \ref{ass}, if $n$ is even, then
$$
\Ktop_1(Q^\hen_\fm \otimes_V \o{F}, \Z/\ell)  = 0. 
$$
\end{cor}

\begin{proof} 
The assumptions allow us to apply Corollary \ref{cor817c}, giving a strongly convergent 
spectral sequence 
$$
E_2^{p,q} \Longrightarrow \Ktop_{q-p}(Q \otimes_V \o{F}, \Z/\ell)
$$
where 
$$
E_2^{p,q} = 
\begin{cases} 
H^p_\et(Q^\hen_\fm \otimes_V \o{F}, \mu_{\ell}^{\otimes i}) & \text{ if $q = 2i$ and} \\
0 & \text{ if $q$ is odd}
\end{cases}
$$
and the differential on the $r$-th page has bidegree $(r,r-1)$.
Since $\o{F}$ is algebraically closed, $\mu_\ell \cong \Z/\ell$ (non-canonically)
and thus Illusie's Theorem applies to give that the only non-zero $E_2$-terms are 
$E_2^{n,2i}$ and $E_2^{0,2i}$. Since $n$ is even, these terms only contribute to the even degree part of $K_*^\top$.
\end{proof}

\begin{rem}
The proof also shows that, when $n$ is even, there exists an exact sequence 
$$
0 \to H^n_\et(Q^\hen_\fm \otimes_V \o{F}, \Z/\ell) \to  
\Ktop_0(Q^\hen_\fm \otimes_V \o{F}, \Z/\ell) \to
H^0_\et(Q^\hen_\fm \otimes_V \o{F}, \Z/\ell)   \to 0
$$
Similarly, when $n$ is odd, there exists an exact sequence 
$$
\begin{aligned}
0 \to 
\Ktop_0(Q^\hen_\fm \otimes_V \o{F}, & \Z/\ell) \to 
H^0_\et(Q^\hen_\fm \otimes_V \o{F}, \Z/\ell)  \\
& \to 
H^n_\et(Q^\hen_\fm \otimes_V \o{F}, \Z/\ell) \to
\Ktop_1(Q^\hen_\fm \otimes_V \o{F}, \Z/\ell) \to 0.
\end{aligned}
$$
\end{rem}

\begin{thm} \label{MainThm}
If conditions \eqref{Qass1}--\eqref{Qass5} of Assumptions \ref{ass} hold and $n$ is even, 
then $\theta^R(M,N) = 0$ for all finitely generated $R$-modules $M$ and $N$, where $R = Q/f$.
\end{thm}

\begin{proof} 
The only singular point of $R/f$ is $\fm$ and hence $\theta^R(M,N)$ is defined for all finitely 
generated $R$-modules.

The hypotheses ensure that $Q^\hen_\fm \otimes_V \o{F}$ is a filtered colimit
  of regular (Noetherian) rings with flat transition maps, and so
$$
\Ktop_1(Q^\hen_\fm \otimes_V \o{F}, \Z/\ell) \cong
\Gtop_1(Q^\hen_\fm \otimes_V \o{F}, \Z/\ell).
$$
The theorem is thus an immediate consequence of Corollaries \ref{Cor98a} and \ref{cor89}. 
\end{proof}

\begin{cor} \label{CorMain} 
Let $V$ be a Dedekind domain, $\fn$ be a maximal ideal of $V$ such that $V/\fn$ is 
a perfect field,  and $Q$ be a regular, flat $V$-algebra of finite type.
Assume the singular locus of the morphism $\Spec(Q) \to \Spec(V)$ is a finite set 
$\{\fm_1, \dots, \fm_1\}$ of maximal ideals of $Q$ that lie over $\fn$ and 
that the morphism $\Spec(Q) \to \Spec(V)$ is a complete intersection in an open neighborhood of each of the $\fm_i$'s.

Then $R := Q/f$ is a hypersurface with only isolated singularities and, 
if $\dm(R)$ is even, 
$$
\theta^R(M,N) = 0
$$ 
for all finitely generated $R$-modules $M$ and $N$.
\end{cor}

\begin{proof} 
We may assume $V$ is local and hence a dvr. Then $R$ is the hypersurface $Q/f$, where $f$ is the image in $Q$ of a
chosen uniformizing parameter $t$ of $V$. 
The non-regular locus of $R$ is $\{\fm_1, \dots, \fm_l\}$ and 
we have
$$
\theta^R(M,N) = \sum_{i=1}^l \theta^{R_{\fm_i}}(M_{\fm_i}, N_{\fm_i})
$$
for all finitely generated $R$-modules $M$ and $N$. It suffices to prove $\theta^{R_{\fm_i}} \equiv 0$, for all $i$, and thus 
upon replacing $\Spec(Q)$ by a 
sufficiently small affine open neighborhood of each $\fm_i$, we may assume
that $l = 1$, that $\fm := \fm_1$ is the only singular point of the morphism $\Spec(Q) \to \Spec(V)$, 
and that this morphism is a complete intersection.

Let $V^\sh_\fn$ denote the {\em strict Henselization} of $V$ at its maximal ideal $\fn$. 
Recall from \cite[\S 18.8]{EGA4d} that there is a faithfully flat local ring map $V \to
V^\sh_\fn$, that $\fn^\sh := \fn V^\sh_\fn$ is the maximal ideal of $V^\sh_\fn$, and 
that the induced map on residue fields $V/\fn \into V^\sh_\fn/\fn^\sh$ is a (initially chosen) 
separable closure of $V/\fn$. Since we assume $V/\fn$
is perfect, the residue field of $V^\sh_\fn$ is, in fact, algebraically closed.

Set  $Q' = Q \otimes_V V^\sh_\fn$ and $R' := R \otimes_V V^\sh_\fn = Q'/f'$ where $f'$ is 
the image of $f$ under $Q \to Q'$. 
The fiber of $\Spec(Q') \to \Spec(Q)$ over $\fm$ is 
$$
Q/\fm \otimes_V V^\sh_\fn \cong 
Q/\fm \otimes_{V/\fn} V^\sh/\fn^\sh = Q/\fm \otimes_k \o{k}.
$$
Since $Q$ has finite type over $V$, $k \into Q/\fm$ is a finite field extension. This shows that 
the fiber of 
$\Spec(Q') \to \Spec(Q)$ over $\fm$ consists of a finite number of maximal ideals $\fm'_1, 
\dots, \fm'_n$ of $Q'$. 
Since $\Spec(Q) \sm \fm \to \Spec(V)$ is smooth, so is $\Spec(Q) \sm \{\fm'_1, \dots, 
\fm'_n\} \to \Spec(V^\sh)$.
For each $i$, upon replacing $Q'$ by a suitably small affine open neighborhood of $\fm'_{i}$, 
conditions (1) -- (5) of Assumptions \ref{ass} are met, and thus we have
$\theta^{R'_{\fm_{i}}}(-,-) \equiv 0$  by the Theorem. 

To prove $\theta^R$ vanishes, it suffices to prove the following more general fact: 
if $R$ is a hypersurface ring having only one singular point, $\fm$, and 
there is a flat, local ring map $(R_\fm, \fm) \to (R', \fm')$ such that $R'$ is a local hypersurface 
with an isolated singularity such that  
$\theta^{R'} \equiv 0$, then $\theta^R \equiv 0$. 
To prove this, observe that if  $T$ is a finitely generated $R$-module supported on $\fm$, then 
$$
\len_{R'}(T \otimes_R R') = \len_R(T) \cdot \len_{R'}(R'/\fm_R R').
$$
It follows that for any pair of finitely generated $R$-modules $M$ and $N$, we have
$$
\begin{aligned}
\len_{R'} \Tor_{R'}^i(M \otimes_R R',N \otimes_R R') 
& = \len_{R'} \left(\Tor_R^i(M,N)  \otimes_R R'\right)  \\
& = \len_R(\Tor_R^i(M,N)) \cdot \len_{R'}(R'/\fm_R R') \\
\end{aligned}
$$
for $i \gg 0$, 
and hence
$$
\theta^R(M,N) = \frac{\theta^{R'}(M \otimes_R R', N \otimes_R R')}{\len_{R'}(R'/\fm_R R')} 
= 0.
$$
\end{proof}

Theorem \ref{introthm} from the Introduction follows quickly from the previous Corollary by taking $V = k[t]$: Since $f$ is a non-zero-divisor, the map of
$k$-algebras $k[t] \to Q$ sending $t$ to $f$ is flat.
Since $k$ is perfect, $Q$ is smooth over $k$ and
hence $Q[t]$ is smooth over $k[t]$. 
The morphism 
$\Spec(Q) \to \Spec(V) = \A^1_k$ is thus a complete intersection because $Q \cong Q[t]/(f-t)$.

We can extend our main vanishing result slightly by allowing localizations:

\begin{cor} \label{Cor114}
For any ring $R$ as in Corollary \ref{CorMain}, $\theta^{S^{-1}R}(M,N) = 0$  for 
any multiplicatively closed set $S$ disjoint from the singular locus
  of $R$ and any pair of finitely generated $S^{-1}R$-modules $M$ and $N$. 
\end{cor}

\begin{proof} More generally, suppose $R$ is any hypersurface whose non-regular locus 
is $\{\fm_1, \dots, \fm_t\} \subseteq \mSpec(R)$ and $\theta^R 
\equiv 0$, and  let  $R' = S^{-1}R$ for any multiplicatively closed set
$S$ 
with $S \cap \fm_i = \emptyset$ for
all $i$. We claim that $\theta^{R'}$ too. 
It is clear $R'$ is also a hypersurface with isolated singularities. 
Given finitely generated $R'$-modules $M$ and $N$, there exist finitely generated $R$-
modules $\t{M}$ and $\t{N}$
such that $\t{M} \otimes_R R' \cong M$ and $\t{N} \otimes_R R' \cong N$. For all $i$ we have
$$
S^{-1} \Tor^{R}_i(\t{M},\t{N}) \cong
\Tor^{R'}_i(M,N),
$$
and, for $i \gg 0$, $\Tor^{R}_i(\t{M},\t{N})$ is supported on $\{\fm_1, \dots, \fm_t\}$ so that
$S^{-1} \Tor^{R}_i(\t{M},\t{N}) \cong \Tor^{R}_i(\t{M},\t{N})$. It follows that 
$\theta^{R'}(M,N) = \theta^R(\t{M}, \t{N}) = 0$.
\end{proof}

In particular, Corollary \ref{Cor114} justifies Example \ref{IntroEx} in the Introduction.

\bibliographystyle{amsplain}

\providecommand{\bysame}{\leavevmode\hbox to3em{\hrulefill}\thinspace}
\providecommand{\MR}{\relax\ifhmode\unskip\space\fi MR }
\providecommand{\MRhref}[2]{%
  \href{http://www.ams.org/mathscinet-getitem?mr=#1}{#2}
}
\providecommand{\href}[2]{#2}


\begin{thebibliography}{10}

\bibitem{BvS}
Ragnar-Olaf Buchweitz and Duco Van~Straten, \emph{An index theorem for modules
  on a hypersurface singularity}, Mosc. Math. J. \textbf{12} (2012), no.~2,
  237--259, 459. \MR{2978754}

\bibitem{DaoDecent}
Hailong Dao, \emph{Decent intersection and {T}or-rigidity for modules over
  local hypersurfaces}, Trans. Amer. Math. Soc. \textbf{365} (2013), no.~6,
  2803--2821. \MR{3034448}

\bibitem{GabberHenelian}
Ofer Gabber, \emph{{$K$}-theory of {H}enselian local rings and {H}enselian
  pairs}, Algebraic {$K$}-theory, commutative algebra, and algebraic geometry
  ({S}anta {M}argherita {L}igure, 1989), Contemp. Math., vol. 126, Amer. Math.
  Soc., Providence, RI, 1992, pp.~59--70. \MR{1156502 (93c:19005)}

\bibitem{GG}
Henri Gillet and Daniel~R. Grayson, \emph{The loop space of the
  {$Q$}-construction}, Illinois J. Math. \textbf{31} (1987), no.~4, 574--597.
  \MR{909784 (89h:18012)}

\bibitem{GraysonAuto}
Daniel~R. Grayson, \emph{{$K_{2}$} and the {$K$}-theory of automorphisms}, J.
  Algebra \textbf{58} (1979), no.~1, 12--30. \MR{535839 (80g:16029)}

\bibitem{EGA4d}
A.~Grothendieck, \emph{\'{E}l\'ements de g\'eom\'etrie alg\'ebrique. {IV}.
  \'{E}tude locale des sch\'emas et des morphismes de sch\'emas {IV}}, Inst.
  Hautes \'Etudes Sci. Publ. Math. (1967), no.~32, 361. \MR{0238860 (39 \#220)}

\bibitem{Hochster}
Melvin Hochster, \emph{The dimension of an intersection in an ambient
  hypersurface}, Algebraic geometry ({C}hicago, {I}ll., 1980), Lecture Notes in
  Math., vol. 862, Springer, Berlin, 1981, pp.~93--106.

\bibitem{Illusie}
Luc Illusie, \emph{Perversit\'e et variation}, Manuscripta Math. \textbf{112}
  (2003), no.~3, 271--295. \MR{2067039 (2005i:14021)}

\bibitem{MilnorFiber}
John Milnor, \emph{Singular points of complex hypersurfaces}, Annals of
  Mathematics Studies, No. 61, Princeton University Press, Princeton, N.J.,
  1968. \MR{0239612 (39 \#969)}

\bibitem{MilnorKtheory}
\bysame, \emph{Introduction to algebraic {$K$}-theory}, Princeton University
  Press, Princeton, N.J., 1971, Annals of Mathematics Studies, No. 72.
  \MR{0349811 (50 \#2304)}

\bibitem{NenashevK1GenRel}
A.~Nenashev, \emph{{$K_1$} by generators and relations}, J. Pure Appl. Algebra
  \textbf{131} (1998), no.~2, 195--212. \MR{1637539 (99g:19001)}

\bibitem{NenashevDSES}
Alexander Nenashev, \emph{Double short exact sequences produce all elements of
  {Q}uillen's {$K_1$}}, Algebraic {$K$}-theory ({P}ozna\'n, 1995), Contemp.
  Math., vol. 199, Amer. Math. Soc., Providence, RI, 1996, pp.~151--160.
  \MR{1409623 (97g:19001)}

\bibitem{OR}
Andreas Rosenschon and P.~A. {\O}stv{\ae}r, \emph{Descent for {$K$}-theories},
  J. Pure Appl. Algebra \textbf{206} (2006), no.~1-2, 141--152. \MR{2220086
  (2007d:19005)}

\bibitem{Serre}
Jean-Pierre Serre, \emph{Cohomologie galoisienne}, fifth ed., Lecture Notes in
  Mathematics, vol.~5, Springer-Verlag, Berlin, 1994. \MR{1324577 (96b:12010)}

\bibitem{ShermanK1Exact}
Clayton Sherman, \emph{On {$K_1$} of an exact category}, $K$-Theory \textbf{14}
  (1998), no.~1, 1--22. \MR{1621689 (99e:19002)}

\bibitem{ShermanConnII}
\bysame, \emph{Connecting homomorphisms in localization sequences. {II}},
  $K$-Theory \textbf{32} (2004), no.~4, 365--389. \MR{2112903 (2005k:19004)}

\bibitem{ShermanMirror}
\bysame, \emph{{$K_1$} of exact categories by mirror image sequences}, J.
  K-Theory \textbf{11} (2013), no.~1, 155--181. \MR{3034287}

\bibitem{Suslin}
A.~Suslin, \emph{On the {$K$}-theory of algebraically closed fields}, Invent.
  Math. \textbf{73} (1983), no.~2, 241--245. \MR{714090 (85h:18008a)}

\bibitem{SuslinLocalFields}
Andrei~A. Suslin, \emph{On the {$K$}-theory of local fields}, Proceedings of
  the {L}uminy conference on algebraic {$K$}-theory ({L}uminy, 1983), vol.~34,
  1984, pp.~301--318. \MR{772065 (86d:18010)}

\bibitem{Thomason}
R.~W. Thomason, \emph{Algebraic {$K$}-theory and \'etale cohomology}, Ann. Sci.
  \'Ecole Norm. Sup. (4) \textbf{18} (1985), no.~3, 437--552. \MR{826102
  (87k:14016)}

\bibitem{WeibelProducts}
C.~A. Weibel, \emph{A survey of products in algebraic {$K$}-theory}, Algebraic
  {$K$}-theory, {E}vanston 1980 ({P}roc. {C}onf., {N}orthwestern {U}niv.,
  {E}vanston, {I}ll., 1980), Lecture Notes in Math., vol. 854, Springer,
  Berlin, 1981, pp.~494--517. \MR{618318 (82k:18011)}

\end{thebibliography}

\end{document}